\documentclass[11pt,leqno]{amsart}
\usepackage{amssymb,amsmath,amsthm,latexsym}
\usepackage{graphicx}
\usepackage{epsfig}
\newtheorem{theorem}{Theorem}[section]
\newtheorem{lemma}[theorem]{Lemma}

\theoremstyle{remark}
\newtheorem{definition}[theorem]{Definition}
\newtheorem{remark}[theorem]{Remark}
\newtheorem{example}[theorem]{Example}

\newcommand\eL{{\mathcal L}}
\newcommand\M{{\mathcal M}}
\newcommand\R{{\mathcal R}}

\newcommand\Z{{\mathbb Z}}
\newcommand\xt{{(x_0,x_1,\ldots,x_n,x_{n+1})}}

\begin{document}

\title{Homology of ternary algebras yielding invariants of knots and knotted surfaces}
\date{January 6, 2017}
\author{Maciej Niebrzydowski}
\address[Maciej Niebrzydowski]{Institute of Mathematics\\ 
Faculty of Mathematics, Physics and Informatics\\
University of Gda{\'n}sk, 80-308 Gda{\'n}sk, Poland}
\email{mniebrz@gmail.com}


\thispagestyle{empty}

\begin{abstract}
We define homology of ternary algebras satisfying axioms derived from particle scattering or, equivalently, from the third Reidemeister move. We show that ternary quasigroups satisfying these axioms appear naturally in invariants of Reidemeister, Yoshikawa, and Roseman moves. Our homology has a degenerate subcomplex. The normalized homology yields invariants of knots and knotted surfaces.
\end{abstract}

\maketitle

\section{Introduction}
The fundamental group of the complement of a knot $K$ in $\mathbb{R}^3$ is often considered either via the Wirtinger relations (of the form $x_i=x_k^{-1}x_jx_k$) or with the Dehn presentation. The binary operations in structures called racks and quandles (\cite{FR92,Joy,Mat82}), and their usefulness in knot theory,  arise from generalizing the conjugation in the Wirtinger relations.
Rack and quandle (co)homology has been closely studied in recent years, see for example \cite{FRS04,CJKLS03,CES04,PR14}. The cocycle invariants obtained from such (co)homology theories proved to be very useful. Some problems to which they were applied are:
the tangle embedding problem \cite{AERS08},
showing non-invertibility of knotted surfaces \cite{CJKLS03,AS05}, calculating the minimal number of triple points in knotted surface projections \cite{SS04}, and finding the
minimal number of broken sheets in knotted surface diagrams \cite{SaS05}.

In this paper we focus on (co)homology of structures obtained from generalizing the relations in the Dehn presentation of the knot group. The relations are of the form $d=ab^{-1}c$ and can be viewed as $d=abcT$, where $abcT=ab^{-1}c$ is a ternary operation that appears quite often in the universal-algebraic literature, e.g. in \cite{Cer43}. This point of view leads us to ternary quasigroups satisfying two axioms obtained from the third Reidemeister move (Definition \ref{KTQ}). We call such a structure a knot-theoretic ternary
quasigroup (abbreviated to KTQ). Our generalization has two stages.
First, we use unoriented diagrams (remembering that Dehn presentation does not require a diagram to be oriented), and then we consider more general algebras involving orientation. The number of KTQ-colorings
of a knot diagram (resp. Yoshikawa diagram or a knotted surface diagram in $\mathbb{R}^3$) does not change under Reidemeister moves (resp. Yoshikawa or Roseman moves), thus it is a property of an isotopy class of a knot (resp. knotted surface) and not just diagram (Section 3).

The main part of the paper presents a homology for algebras $(X,T)$ satisfying the nesting axioms A1 and A2 derived from the third Reidemeister move, and the corresponding degenerate subcomplex. We define the homology of KTQs as the normalized homology $H^N(X,T)$. We show how to assign a cycle in this homology to a KTQ-colored diagram of a knot (resp. knotted surface), so that its homology class is a knot (resp. knotted surface) invariant. 

$H^N_*(X,T)$ very often has a torsion part. Our calculations (for homology with $\mathbb{Z}$ coefficients) with GAP \cite{GAP4} indicate that, up to isomorphism, there are 2 two-element KTQs, of which one has a torsion part $\mathbb{Z}_2$ in $H^N_0(X,T)$ and $H^N_2(X,T)$. There are 7 KTQs with three elements, and out of them five have torsion (either  
$\mathbb{Z}_3$ or $\mathbb{Z}_3^2$) in $H^N_2(X,T)$. There are 37 non-isomorphic four-element KTQs,
of which only three have no torsion part in $H^N_2(X,T)$, and for the rest the possibilities are: $\mathbb{Z}_2$, $\mathbb{Z}_2^2$, $\mathbb{Z}_2^3$, $\mathbb{Z}_2^4$, $\mathbb{Z}_2\oplus \mathbb{Z}_4^2$, $\mathbb{Z}_2^2\oplus \mathbb{Z}_4^2$, $\mathbb{Z}_2^3\oplus \mathbb{Z}_4$, and $\mathbb{Z}_2^3\oplus \mathbb{Z}_4^2$. There are 23 KTQs with five elements; five of them have
torsion part (either $\mathbb{Z}_5$ or $\mathbb{Z}_5^2$) in $H^N_2(X,T)$. Finally, all 193 non-isomorphic six-element KTQs that we were able to generate so far with GAP have nontrivial torsion part in the second homology $H^N_2(X,T)$, and only 14 have no torsion in $H^N_1(X,T)$. Note that $H^N_1(X,T)$ is used for invariants of 1-knots, and $H^N_2(X,T)$ is for 2-knots.

Some connections (requiring certain assumptions) between arc colorings and region colorings are considered in \cite{Hi94}. We will present an example of a knot diagram on a torus for which the fundamental (shadow) quandle and cocycle invariants do not work, but KTQs with the associated cohomological invariants can be applied.

The rest of the paper is organized as follows. In Section 2 we construct the homology for structures satisfying the axioms A1 and A2, and the degenerate subcomplex. In Section 3 we review the definition of a ternary quasigroup and  explain how KTQs can be used in invariants of knots and knotted surfaces. Finally, in Section 4 we construct (co)homological invariants and show some examples of computations.

\section{Homology} 

As in \cite{Hi94}, we begin with some motivation from physics.
We consider three particles moving with different velocities in one-dimensional ambient space. They divide it into parts and the state of the vacuum can be different in them (see Fig. \ref{particles}). When two particles approach each other, they scatter and recede from each other preserving momenta, but the state of the vacuum between them can change, and we will assume that the new state is described as $abcT$, where $T\colon X\times X\times X\to X$ is a ternary operation on the set $X$ of states, and $a$, $b$, and $c$ are the states before scattering, taken in a cyclic clockwise order as in Fig. \ref{scattering}. With three particles, there will be exactly three pairwise scatterings, but their order depends on the initial position of the particles. It is a natural assumption that the states of the vacua after all pairwise scatterings should not depend on this order, and thus two axioms are obtained:
\begin{equation}
\forall_{a,b,c,d\in X} \quad (abcT)cdT=[ab(bcdT)T](bcdT)dT, \tag{A1} 
\end{equation}
\begin{equation}
\forall_{a,b,c,d\in X} \quad ab(bcdT)T=a(abcT)[(abcT)cdT]T.  \tag{A2}
\end{equation}

Note that the right side of A1 (resp. A2) is obtained from the left side of A1 (resp. A2) by substitution $c\mapsto bcdT$ (resp. $b\mapsto abcT$ ).

\begin{figure}
\begin{center}
\includegraphics[height=0.8 cm]{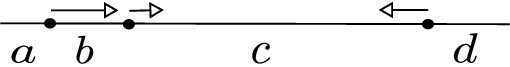}
\caption{Particles moving in one-dimensional ambient space.}\label{particles}
\end{center}
\end{figure}

\begin{figure}
\begin{center}
\includegraphics[height=6 cm]{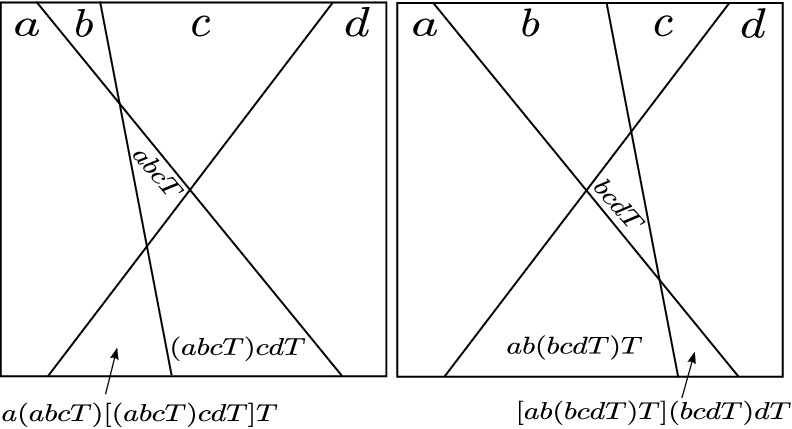}
\caption{Particle scatterings and the states of the vacua.}\label{scattering}
\end{center}
\end{figure}

\begin{example}
Let $(G,\cdot)$ be a group. Consider a generalization of conjugation:
\[
abcS=a^{-1}\cdot b\cdot c,\ abc\overline{S}=a\cdot b\cdot c^{-1}.
\]
Then $S$ satisfies A2 but not A1, and $\overline{S}$ satisfies A1 but not A2.
\end{example}
\begin{example}
Let $(G,\cdot)$ be a group. The operation from the Dehn presentation
\[
abcT=a\cdot b^{-1}\cdot c
\]
satisfies both A1 and A2.
\end{example}

Now we will define homology theory for algebras $(X,T)$ satisfying A1 and A2. 

\begin{definition} \label{maindefs}
Let $\mathrm{R}$ be a commutative unital ring, and
$(X,T)$ denote a ternary algebra satisfying axioms A1 and A2.
Let $C_n(X)=\mathrm{R}\langle X^{n+2}\rangle$ be the $\mathrm{R}$-module generated freely by $(n+2)$-tuples 
$(x_0,x_1,\ldots, x_n,x_{n+1})$ of elements of $X$. We define
\[
\partial_n^L\xt=\sum_{i=0}^n(-1)^i d_i^{n,L}\xt,
\]
where $d_i^{n,L}\colon C_n\to C_{n-1}$ is obtained inductively by
\begin{align*}
d_0^{n,L}\xt & =(x_1,\ldots,x_{n+1}),\\
d_{i}^{n,L}\xt & = d_{i-1}^{n,L}(x_0,\ldots,x_{i-1},x_{i-1}x_ix_{i+1}T,x_{i+1},\ldots,x_{n+1})
\end{align*}
for $i\in\{1,\ldots,n\}$. Similarly,
\[
\partial_n^R\xt=\sum_{i=0}^n(-1)^i d_i^{n,R}\xt,
\]
where the inductive formula for $d_i^{n,R}\colon C_n\to C_{n-1}$ is as follows:
\begin{align*}
d_n^{n,R}\xt & =(x_0,\ldots,x_{n}),\\
d_{i-1}^{n,R}\xt & =d_{i}^{n,R}(x_0,\ldots,x_{i-1},x_{i-1}x_ix_{i+1}T,x_{i+1},\ldots,x_{n+1})
\end{align*}
for $i\in\{1,\ldots,n\}$. 

We can describe $d_i^{n,L}$ and $d_i^{n,R}$ in a different way, defining their coordinates inductively. 
$d_i^{n,L}=(d_{i,1}^{n,L},\ldots,d_{i,k}^{n,L},\ldots,d_{i,n+1}^{n,L})$
is calculated from right to left. For $i\in\{0,\ldots,n\}$ and $k\in\{1,\ldots,n+1\}$,
\begin{equation}\label{leftcoord}
d_{i,k}^{n,L}x = \left\{
\begin{array}{rl}
x_{k-1}x_k(d_{i,k+1}^{n,L}x)T & \text{if } k\leq  i\\
x_k & \text{if } k > i.
\end{array} \right.
\end{equation}
$d_i^{n,R}=(d_{i,0}^{n,R},\ldots,d_{i,k}^{n,R},\ldots,d_{i,n}^{n,R})$ 
is calculated from left to right. For $i\in\{0,\ldots,n\}$ and $k\in\{0,\ldots,n\}$,
\begin{equation}\label{rightcoord}
d_{i,k}^{n,R}x = \left\{
\begin{array}{rl}
(d_{i,k-1}^{n,R}x)x_kx_{k+1}T & \text{if } k > i\\
x_k & \text{if } k \leq i.
\end{array} \right.
\end{equation}
\end{definition}

\begin{theorem} \label{mainhom}
The $\mathrm{R}$-modules $C_n$ endowed with maps
\[
\partial_n^{(\alpha,\beta)}=\alpha\partial_n^L
+\beta\partial_n^R
\]
form a chain complex for any $\alpha$, $\beta\in \mathrm{R}$.
\end{theorem}
\begin{proof}
We need to show that
\[
d_i^{n-1,\epsilon} d_j^{n,\delta}=d_{j-1}^{n-1,\delta} d_i^{n,\epsilon}
\]
for all $i<j$ and $\epsilon$, $\delta\in\{L,R\}$. It will be done in a series of lemmas below.
\end{proof}

Let $x$ denote $\xt$.

\begin{lemma} \label{firstmod}
$d_i^{n-1,L}d_j^{n,L}x=d_{j-1}^{n-1,L}d_i^{n,L}x$
for $0\leq i<j\leq n$.
\end{lemma}
\begin{proof}
The proof is by induction over $j-i$.
First, let $j-i=1$, so we need 
\begin{equation}
d_i^{n-1,L}d_{i+1}^{n,L}x=d_i^{n-1,L}d_i^{n,L}x. \label{ij1}
\end{equation}
For $i=0$, the validity of equation (\ref{ij1}) is immediate ($d_0^{*,L}$ just removes the first input), so assume that $i>0$.
Until the end of this proof, we denote $(d_{i+1,1}^{n,L}x,\ldots,d_{i+1,k}^{n,L}x,
\ldots,d_{i+1,n+1}^{n,L}x)$ by $(i+1)$, and
$(d_{i,1}^{n,L}x,\ldots,d_{i,k}^{n,L}x,
\ldots,d_{i,n+1}^{n,L}x)$ by $(i)$.
Then
\begin{align*}
& d_i^{n-1,L}d_{i+1}^{n,L}x=
d_i^{n-1,L}(d_{i+1,1}^{n,L}x,\ldots,d_{i+1,k}^{n,L}x,
\ldots,d_{i+1,n+1}^{n,L}x)=d_i^{n-1,L}(i+1)
\\ & =(d_{i,1}^{n-1,L}(i+1),\ldots,d_{i,k}^{n-1,L}(i+1),\ldots,d_{i,n}^{n-1,L}(i+1)),\\
& d_i^{n-1,L}d_i^{n,L}x=d_i^{n-1,L}(d_{i,1}^{n,L}x,\ldots,d_{i,k}^{n,L}x,\ldots,
d_{i,n+1}^{n,L}x)=d_i^{n-1,L}(i)
\\ & =(d_{i,1}^{n-1,L}(i),\ldots,d_{i,k}^{n-1,L}(i),\ldots,d_{i,n}^{n-1,L}(i)).
\end{align*}
We see that for $k\geq i+1$,
\[
d_{i,k}^{n-1,L}(i+1)=d_{i+1,k+1}^{n,L}x=x_{k+1}=d_{i,k+1}^{n,L}x=d_{i,k}^{n-1,L}(i).
\]
For $k=i$,
\[
d_{i,k}^{n-1,L}(i)=(d_{i,i}^{n,L}x)(d_{i,i+1}^{n,L}x)(d_{i,i+1}^{n-1,L}(i))T=
(x_{i-1}x_ix_{i+1}T)x_{i+1}x_{i+2}T
\]
and
\begin{align*}
d_{i,k}^{n-1,L}(i+1) & =(d_{i+1,i}^{n,L}x)(d_{i+1,i+1}^{n,L}x)(d_{i,i+1}^{n-1,L}(i+1))T\\
& =[x_{i-1}x_i(x_ix_{i+1}x_{i+2}T)T](x_ix_{i+1}x_{i+2}T)x_{i+2}T.
\end{align*}
Thus, the equality of $d_{i,i}^{n-1,L}(i)$ and $d_{i,i}^{n-1,L}(i+1)$
is exactly the application of the axiom A1. To prove the equalities of the remaining coordinates, we use the axiom A2. Note that 
\[
d_{i+1,i+1}^{n,L}x=x_ix_{i+1}x_{i+2}T=x_ix_{i+1}(d_{i,i+1}^{n-1,L}(i))T.
\]
In general, for $k\leq i+1$, there is a relation
\[
d_{i+1,k}^{n,L}x=x_{k-1}(d_{i,k}^{n,L}x)(d_{i,k}^{n-1,L}(i))T.
\]
We prove it by induction, using A2:
\begin{align*}
& d_{i+1,k-1}^{n,L}x=x_{k-2}x_{k-1}(d_{i+1,k}^{n,L}x)T=x_{k-2}x_{k-1}[x_{k-1}(d_{i,k}^{n,L}x)(d_{i,k}^{n-1,L}(i))T]T\\
& =x_{k-2}(x_{k-2}x_{k-1}(d_{i,k}^{n,L}x)T)[(x_{k-2}x_{k-1}(d_{i,k}^{n,L}x)T)
(d_{i,k}^{n,L}x)(d_{i,k}^{n-1,L}(i))T]T\\
& =x_{k-2}(d_{i,k-1}^{n,L}x)[(d_{i,k-1}^{n,L}x)(d_{i,k}^{n,L}x)(d_{i,k}^{n-1,L}(i))T]T=x_{k-2}(d_{i,k-1}^{n,L}x)(d_{i,k-1}^{n-1,L}(i))T.
\end{align*}

Next, we show the equality of the rest of the coordinates in $d_i^{n-1,L}(i+1)$ and 
$d_i^{n-1,L}(i)$, using induction and A1:
\begin{align*}
& d_{i,k-1}^{n-1,L}(i+1)=(d_{i+1,k-1}^{n,L}x)(d_{i+1,k}^{n,L}x)(d_{i,k}^{n-1,L}(i+1))T
\\ & =(d_{i+1,k-1}^{n,L}x)(d_{i+1,k}^{n,L}x)(d_{i,k}^{n-1,L}(i))T
=[x_{k-2}x_{k-1}(d_{i+1,k}^{n,L}x)T](d_{i+1,k}^{n,L}x)(d_{i,k}^{n-1,L}(i))T\\
& =\{x_{k-2}x_{k-1}[x_{k-1}(d_{i,k}^{n,L}x)(d_{i,k}^{n-1,L}(i))T]T\}[x_{k-1}(d_{i,k}^{n,L}x)(d_{i,k}^{n-1,L}(i))T](d_{i,k}^{n-1,L}(i))T\\
& =(x_{k-2}x_{k-1}(d_{i,k}^{n,L}x)T)(d_{i,k}^{n,L}x)(d_{i,k}^{n-1,L}(i))T
=(d_{i,k-1}^{n,L}x)(d_{i,k}^{n,L}x)(d_{i,k}^{n-1,L}(i))T\\
& =d_{i,k-1}^{n-1,L}(i).
\end{align*}

Now assume that  
\[
d_{i'}^{n-1,L}d_{j'}^{n,L}x=d_{j'-1}^{n-1,L}d_{i'}^{n,L}x
\]
for $i'$, $j'$ such that $0\leq i'<j'\leq n$ and $j'-i'<j-i$, where $j-i\geq 2$.
For the rest of the paper, let $x[k]$ denote $(x_0,\ldots,x_{k-1}x_kx_{k+1}T,\ldots,x_{n+1})$. With this new notation $d_i^{n,L}x=d_{i-1}^{n,L}x[i]$
for $i\in\{1,\ldots,n\}$. We have
\begin{align*}
& d_i^{n-1,L}d_j^{n,L}x=d_i^{n-1,L}d_{j-1}^{n,L}x[j]
=d_{j-2}^{n-1,L}d_i^{n,L}x[j]\\
& =d_{j-2}^{n-1,L}(d_{i,1}^{n,L}x[j],\ldots,d_{i,j-2}^{n,L}x[j],d_{i,j-1}^{n,L}x[j],d_{i,j}^{n,L}x[j],d_{i,j+1}^{n,L}x[j],\ldots,d_{i,n+1}^{n,L}x[j])\\
& =d_{j-2}^{n-1,L}(d_{i,1}^{n,L}x[j],\ldots,d_{i,j-2}^{n,L}x[j],x_{j-1},x_{j-1}x_jx_{j+1}T,x_{j+1},\ldots,x_{n+1})
\end{align*}
and
\begin{align*}
& d_{j-1}^{n-1,L}d_i^{n,L}x=d_{j-1}^{n-1,L}(d_{i,1}^{n,L}x,\ldots,d_{i,j-1}^{n,L}x,d_{i,j}^{n,L}x,d_{i,j+1}^{n,L}x,\ldots,d_{i,n+1}^{n,L}x)
\\ & =d_{j-2}^{n-1,L}(d_{i,1}^{n,L}x,\ldots,d_{i,j-1}^{n,L}x,(d_{i,j-1}^{n,L}x)(d_{i,j}^{n,L}x)(d_{i,j+1}^{n,L}x)T,d_{i,j+1}^{n,L}x,\ldots,d_{i,n+1}^{n,L}x)\\
& =d_{j-2}^{n-1,L}(d_{i,1}^{n,L}x,\ldots,d_{i,j-2}^{n,L}x,x_{j-1},x_{j-1}x_jx_{j+1}T,x_{j+1},\ldots,x_{n+1}).
\end{align*}
Now the proof ends, since $d_{i,j-1}^{n,L}x=d_{i,j-1}^{n,L}x[j]=x_{j-1}$,
so from the formula (\ref{leftcoord}), for $k\leq j-2$, it follows that
\begin{align*}
& d_{i,k}^{n,L}(x_0,\ldots,x_{j-2},x_{j-1},x_{j-1}x_jx_{j+1}T,x_{j+1},\ldots,x_{n+1})
\\ & =d_{i,k}^{n,L}(x_0,\ldots,x_{j-2},x_{j-1},x_j,x_{j+1},\ldots,x_{n+1}).
\end{align*}
\end{proof}

\begin{definition}
Given a ternary operation $T$, let $\hat{T}$ denote the ternary operation
defined by $xyz\hat{T}=zyxT$.
\end{definition}

\begin{remark}\label{rev1}
$(X,T)$ satisfies A2 if and only if $(X,\hat{T})$ satisfies A1.\\
$(X,T)$ satisfies A1 if and only if $(X,\hat{T})$ satisfies A2.
\end{remark}

Let $y^r$ denote reversing the order of the elements of the tuple $y$; we will also use the linear extension of this operator, denoting it with the same symbol.
When two or more operators are considered, we will add their symbols to the differentials, as in the following lemma.
\begin{lemma}\label{convert}
$d_i^{n,R,T}x=(d_{n-i}^{n,L,\hat{T}}x^r)^r$ for $i\in\{0,\ldots,n\}$.
\end{lemma}
We leave a simple inductive proof to the reader. 

\begin{lemma} \label{secondmod}
$d_i^{n-1,R}d_j^{n,R}x=d_{j-1}^{n-1,R}d_i^{n,R}x$
for $0\leq i<j\leq n$.
\end{lemma}
\begin{proof}
If $T$ satisfies the axioms A1 and A2, then $\hat{T}$ satisfies them also, and the equation 
\[
d_i^{n-1,L,\hat{T}}d_j^{n,L,\hat{T}}=d_{j-1}^{n-1,L,\hat{T}}d_i^{n,L,\hat{T}}
\]
holds for $0\leq i<j\leq n$. Then, for $0\leq i<j\leq n$, we have
\begin{align*}
d_i^{n-1,R,T}d_j^{n,R,T}x & =d_i^{n-1,R,T}(d_{n-j}^{n,L,\hat{T}}x^r)^r
=(d_{n-1-i}^{n-1,L,\hat{T}}((d_{n-j}^{n,L,\hat{T}}x^r)^r)^r)^r\\
& =(d_{n-i-1}^{n-1,L,\hat{T}}d_{n-j}^{n,L,\hat{T}}x^r)^r
=(d_{n-j}^{n-1,L,\hat{T}}d_{n-i}^{n,L,\hat{T}}x^r)^r\\
& =(d_{n-1-(j-1)}^{n-1,L,\hat{T}}((d_{n-i}^{n,L,\hat{T}}x^r)^r)^r)^r
=(d_{n-1-(j-1)}^{n-1,L,\hat{T}}(d_i^{n,R,T}x)^r)^r\\
& =d_{j-1}^{n-1,R,T}d_i^{n,R,T}x.
\end{align*} 
\end{proof}

\begin{lemma}
$d_i^{n-1,R}d_j^{n,L}=d_{j-1}^{n-1,L}d_i^{n,R}$
for $0\leq i<j\leq n$.
\end{lemma}
\begin{proof}
The proof is by induction over $j-i$. First,
\[
d_i^{n-1,R}d_{i+1}^{n,L}=d_i^{n-1,L}d_i^{n,R},
\] 
for $i\in\{0,\ldots,n-1\}$,
follows from the equalities 
\[
d_{i,k-1}^{n-1,R}(d_{i+1,1}^{n,L}x,\ldots,d_{i+1,n+1}^{n,L}x)
=d_{i,k}^{n-1,L}(d_{i,0}^{n,R}x,\ldots,d_{i,n}^{n,R}x),
\]
that are true for $k\in\{1,\ldots,n\}$.
By the inductive step (with $j-1>i$):
\begin{align*}
& d_i^{n-1,R}d_j^{n,L}x=d_i^{n-1,R}d_{j-1}^{n,L}x[j]=d_{j-2}^{n-1,L}d_i^{n,R}
x[j]\\
& =d_{j-2}^{n-1,L}(d_{i,0}^{n,R}x[j],\ldots,d_{i,j-2}^{n,R}x[j],d_{i,j-1}^{n,R}x[j],
d_{i,j}^{n,R}x[j],\ldots,d_{i,n}^{n,R}x[j]),\\
& d_{j-1}^{n-1,L}d_i^{n,R}x
=d_{j-2}^{n-1,L}(d_{i,0}^{n,R}x,\ldots,d_{i,j-2}^{n,R}x,
(d_{i,j-2}^{n,R}x)(d_{i,j-1}^{n,R}x)(d_{i,j}^{n,R}x)T,\ldots,
d_{i,n}^{n,R}x).
\end{align*}
We check that the inputs for $d_{j-2}^{n-1,L}$ are equal in both expressions.
First, by comparing $d_{i,k}^{n,R}x[j]$ with $d_{i,k}^{n,R}x$ for $k\leq j-2$,
we see that the coordinate that makes a difference between $x$ and $x[j]$ is not used. Consider the $j$-th inputs:
\begin{align*}
& (d_{i,j-2}^{n,R}x)(d_{i,j-1}^{n,R}x)(d_{i,j}^{n,R}x)T\\
&=(d_{i,j-2}^{n,R}x)((d_{i,j-2}^{n,R}x)x_{j-1}x_jT)[((d_{i,j-2}^{n,R}x)x_{j-1}x_jT)x_{j}x_{j+1}T]T\\
& =(d_{i,j-2}^{n,R}x)x_{j-1}(x_{j-1}x_jx_{j+1}T)T=d_{i,j-1}^{n,R}x[j].
\end{align*}
For the $(j+1)$-st inputs, we have
\begin{align*}
d_{i,j}^{n,R}x[j] & =(d_{i,j-1}^{n,R}x[j])(x_{j-1}x_jx_{j+1}T)x_{j+1}T
\\ & =[(d_{i,j-2}^{n,R}x[j])x_{j-1}(x_{j-1}x_jx_{j+1}T)T](x_{j-1}x_jx_{j+1}T)x_{j+1}T
\\ & =((d_{i,j-2}^{n,R}x[j])x_{j-1}x_jT)x_jx_{j+1}T
=((d_{i,j-2}^{n,R}x)x_{j-1}x_jT)x_jx_{j+1}T \\
& =(d_{i,j-1}^{n,R}x)x_jx_{j+1}T=d_{i,j}^{n,R}x.
\end{align*}
The equalities of the later inputs follow inductively from the equality that we have just checked.
\end{proof}
\begin{lemma}
$d_i^{n-1,L}d_j^{n,R}=d_{j-1}^{n-1,R}d_i^{n,L}$
for $0\leq i<j\leq n$.
\end{lemma}
\begin{proof}
$d_0^{n,L}$ and $d_0^{n-1,L}$ only remove the left-most input, therefore
\[
d_0^{n-1,L}d_j^{n,R}=d_{j-1}^{n-1,R}d_0^{n,L}
\] for $j\in\{1,\ldots,n\}$.
Now we use the induction over $n-(j-i)$ (with 0 corresponding to the case $i=0$, $j=n$ that was just considered):
\begin{align*}
& d_{j-1}^{n-1,R}d_i^{n,L}x=d_{j-1}^{n-1,R}d_{i-1}^{n,L}x[i]
=d_{i-1}^{n-1,L}d_j^{n,R}x[i]\\
& =d_{i-1}^{n-1,L}(d_{j,0}^{n,R}x[i],\ldots,d_{j,i-1}^{n,R}x[i],
d_{j,i}^{n,R}x[i],d_{j,i+1}^{n,R}x[i],\ldots,d_{j,n}^{n,R}x[i]),\\
& d_i^{n-1,L}d_j^{n,R}x
=d_{i-1}^{n-1,L}(d_{j,0}^{n,R}x,\ldots,d_{j,i-1}^{n,R}x,
(d_{j,i-1}^{n,R}x)(d_{j,i}^{n,R}x)(d_{j,i+1}^{n,R}x)T,
\ldots,d_{j,n}^{n,R}x).
\end{align*}
Since $d_{j,k}^{n,R}(y_0,\ldots,y_{n+1})=y_k$ for $k\leq i+1\leq j$, there is equality of the first corresponding $i+2$ inputs for $d_{i-1}^{n-1,L}$ in $d_{j-1}^{n-1,R}d_i^{n,L}x$ and in $d_i^{n-1,L}d_j^{n,R}x$. The equality of the inputs with greater indices follows from the recursive definition of the coordinates of $d_j^{n,R}$.
\end{proof}

Now we define a degenerate subcomplex.

\begin{definition}\label{maindeg}
For $(X,T)$ and $\mathrm{R}$ as before, let $C_n^D(X,T)$ denote the $\mathrm{R}$-module generated freely by $(n+2)$-tuples 
$x=(x_0,x_1,\ldots, x_n,x_{n+1})$ of elements of $X$ with an index $j$, $0\leq j\leq n-1$, such that $x_{j+1}=x_{j}x_{j+1}x_{j+2}T$. For $n<1$, we take $C_n^D(X,T)=0$.
\end{definition}

\begin{lemma}
$\partial_n^{(\alpha,\beta)}(C_n^D(X,T))\subset C_{n-1}^D(X,T)$ for any $\alpha$, $\beta\in \mathrm{R}$.
\end{lemma}
\begin{proof}
Let $x$ be such that $x_{j+1}=x_{j}x_{j+1}x_{j+2}T$, with $0\leq j\leq n-1$.
We will show the inclusion for $\partial_n^L$, the proof for $\partial_n^R$ is completely symmetric. $d_i^{n,L}x$ contains at the end the sequence
$x_{i+1},\ldots,x_{n+1}$. It follows that $x_j$, $x_{j+1}$, $x_{j+2}$ occur also in all
$d_i^{n,L}x$ with $i\in\{0,\ldots, j-1\}$.
Now let $i=j+1$:
\begin{align*}
&d_{j+1}^{n,L}(x_0,\ldots,x_j,x_{j+1},x_{j+2},\ldots,x_{n+1})\\
& =d_j^{n,L}(x_0,\ldots,x_j,x_{j}x_{j+1}x_{j+2}T,x_{j+2},\ldots,x_{n+1})\\
& =d_j^{n,L}(x_0,\ldots,x_j,x_{j+1},x_{j+2},\ldots,x_{n+1}).
\end{align*}
But in $\partial_n^{L}$, $d_{j}^{n,L}$ and $d_{j+1}^{n,L}$ appear with opposite signs. Now let $j+2\leq i \leq n$. We will show that in $d_{i}^{n,L}x$, the triple $d^{n,L}_{i,j+1}x, d^{n,L}_{i,j+2}x, d^{n,L}_{i,j+3}x$ is degenerate. From the formula (\ref{leftcoord}), and the condition A1, it follows that
\begin{align*}
& (d^{n,L}_{i,j+1}x) (d^{n,L}_{i,j+2}x) (d^{n,L}_{i,j+3}x)T \\
& =[x_jx_{j+1}(x_{j+1}x_{j+2}(d^{n,L}_{i,j+3}x)T)T](x_{j+1}x_{j+2}(d^{n,L}_{i,j+3}x)T)(d^{n,L}_{i,j+3}x)T\\
& =(x_jx_{j+1}x_{j+2}T)x_{j+2}(d^{n,L}_{i,j+3}x)T=x_{j+1}x_{j+2}(d^{n,L}_{i,j+3}x)T=d^{n,L}_{i,j+2}x.
\end{align*}
\end{proof}

\section{Ternary quasigroups in knot theory}

In this section we introduce knot-theoretic ternary quasigroups, and show how to obtain from them coloring invariants of knots and knotted surfaces.

\begin{definition}\label{quasigroup}
A {\it ternary quasigroup} is a set $X$ equipped with a ternary operation 
$T\colon X^3\to X$ such that for a quadruple $(x_1,x_2,x_3,x_4)$
of elements of $X$ satisfying $x_1x_2x_3T=x_4$, specification of any three elements of the quadruple determines the remaining one uniquely. This leads to three additional ternary operations $\eL$, $\M$, $\R\colon X^3\to X$, defined via
\[x_4x_2x_3\eL=x_1,\ x_1x_4x_3\M=x_2,\ \textrm{and}\ x_1x_2x_4\R=x_3.\]
We call them the left, middle, and right division, respectively.
\end{definition}

A finite ternary quasigroup $(X,T)$, with elements numbered $1,\ldots,n$, can be described by a Latin cube, i.e., an $n\times n\times n$ array in which every $i\in\{1,\ldots,n\}$ appears exactly once in every horizontal row, every vertical row, and in every column. Any Latin cube defines a ternary quasigroup. 
See \cite{Bel,BelSan,Sm08} for more details on $n$-ary quasigroups.

\begin{definition} \label{KTQ}
A knot-theoretic ternary quasigroup (abbreviated to KTQ) is a ternary quasigroup satisfying the axioms A1 and A2.
\end{definition}

\begin{figure}
\begin{center}
\includegraphics[height=3 cm]{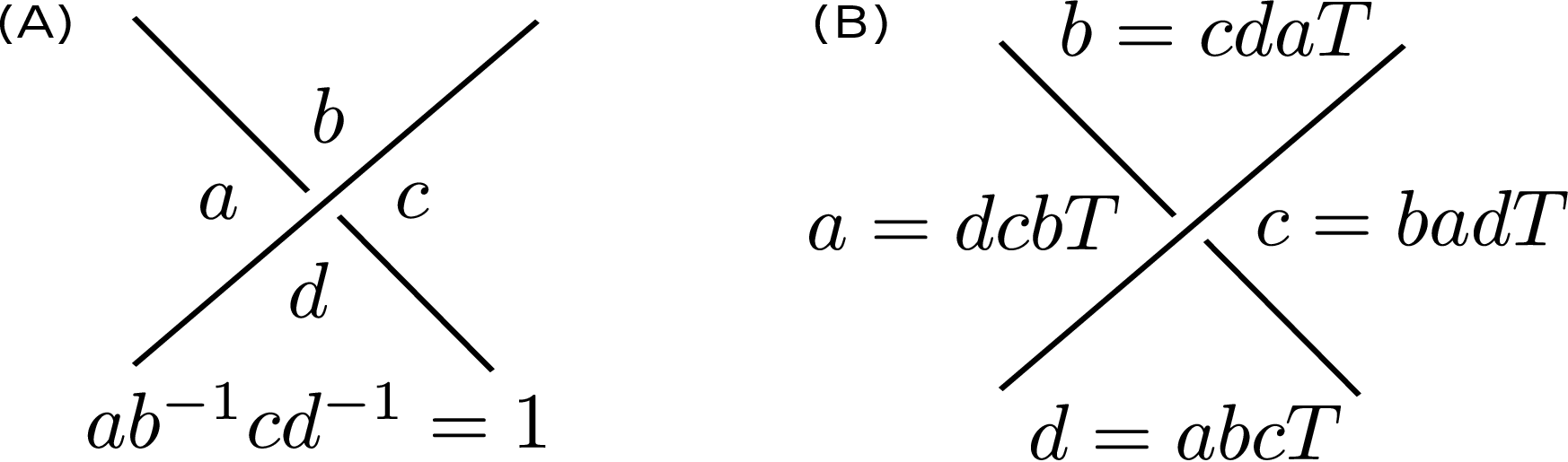}
\caption{A relation in Dehn presentation and its generalization using a ternary quasigroup $(X,T)$.}\label{Dehngen}
\end{center}
\end{figure}

Recall that the fundamental group of the complement of a knot in $\mathbb{R}^3$ can be given the following presentation, called Dehn presentation: generators are assigned to the regions in the complement of a knot diagram $D$ on a plane, and relations correspond to the crossings and are as in Fig. \ref{Dehngen}(A). One of the generators, say the one corresponding to the unbounded region, is set equal to identity. Geometrically, a generator can be viewed as a loop originating from a fixed point $P$ beneath the diagram, piercing a region to which it is assigned, and returning to $P$ through a region labeled with the identity element. See e.g. \cite{Kauffor} for more details about Dehn presentation. 

Note that the fundamental group relations between generators around a crossing can be realized using a ternary operation $xyzT=xy^{-1}z$. Each generator can be expressed using $T$ and the other three generators as in Fig. \ref{Dehngen}(B). Namely, if $x$, $y$, $z$, $w\in X$ are the generators near a crossing, then $w=xyzT$, where $w$ and $x$ are assigned to the regions separated by an over-arc, and $x$, $y$ and $z$ are taken cyclically.

Let $(X,T,\eL,\M,\R)$ be a knot-theoretic ternary quasigroup such that:
\begin{align*} 
& xyz\eL=xyzT^{(2,3)}:=xzyT,\\
& xyz\M=xyz\hat{T}=zyxT,\\
& xyz\R=xyzT^{(1,2)}:=yxzT.
\end{align*}
Then it is easy to check that a group with operation $xyzT=xy^{-1}z$ is an example of such a KTQ. An abstract KTQ of this kind, with generators corresponding to regions in a diagram, and relations of the form $d=abcT$, assigned to crossings as in Fig. \ref{Dehngen}(B), generalizes the knot group.

KTQs of the type $(X,T,T^{(2,3)},\hat{T},T^{(1,2)})$ were used to define knot invariants in \cite{Nie14}.
In \cite{NeNe17} the authors used compositions of two binary quasigroup operations of the form $x*(y\cdot z)$, satisfying the conditions A1 and A2, to define invariants for oriented knots. We also mention the paper \cite{Dev09}, in
which the author constructed combinatorial invariants of knots based on colorings of regions of a knot diagram by elements of some finite ring $\mathrm{R}$, with coloring requirements involving the equation $pa+b-c-pd=0$, for $a$, $b$, $c$, $d\in \mathrm{R}$ and an invertible element $p\in \mathrm{R}$. In this paper, we
show how to use general KTQs $(X,T,\eL,\M,\R)$ for colorings of oriented link diagrams, oriented Yoshikawa diagrams, and oriented knotted surface diagrams in $\mathbb{R}^3$. First, we recall some examples from \cite{Nie14}.

\begin{example}
Let $(X,*)$ be an extra loop. If we define $T$ by $xyzT=(x*y^{-1})*z$, then
$(X,T)$ is a KTQ.
\end{example} 

\begin{example}
Let $(X,*)$ be a Moufang loop. Then $xyzT=(y*x^{-1})*z$ defines a KTQ.
\end{example}

The following three examples are KTQs of the type $(X,T,T^{(2,3)},\hat{T},T^{(1,2)})$, with an additional property that $abaT=b$. This property simplifies the description of degenerate modules $C^D_n(X,T)$: they are generated by $(n+2)$-tuples of elements of $X$ containing $a$, $b$, $a$ on three consecutive coordinates, for some $a$ and $b\in X$.

\begin{example}
$X=\mathbb{R}^n$ with $pqrT=p+q-r$. 
Geometrically, $T$ reflects the point $r$ through the middle of the interval connecting the points $p$ and $q$.
\end{example}

\begin{example}\label{dih}
$X=\{0,\ldots,n-1\}$ with $pqrT=p+q-r \pmod n$. 
\end{example}

More generally:

\begin{example}
A group $(G,\cdot)$ with $xyzT=x\cdot z^{-1}\cdot y$. One can show that this operation yields the relations of the core group of a link, $core(L)$.
From the point of view of algebraic topology, $core(L)$ is the free product of the fundamental group of the cyclic branched double cover of $\mathbb{S}^3$ with branching set $L$ and the infinite cyclic group \cite{FR92, Wa92}.
\end{example}

\begin{figure}
\begin{center}
\includegraphics[height=2.5 cm]{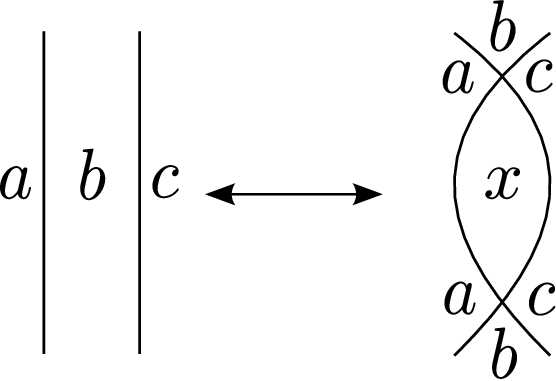}
\caption{A generic second Reidemeister move and its coloring.}\label{R2gen}
\end{center}
\end{figure}

\begin{figure}
\begin{center}
\includegraphics[height=3.7 cm]{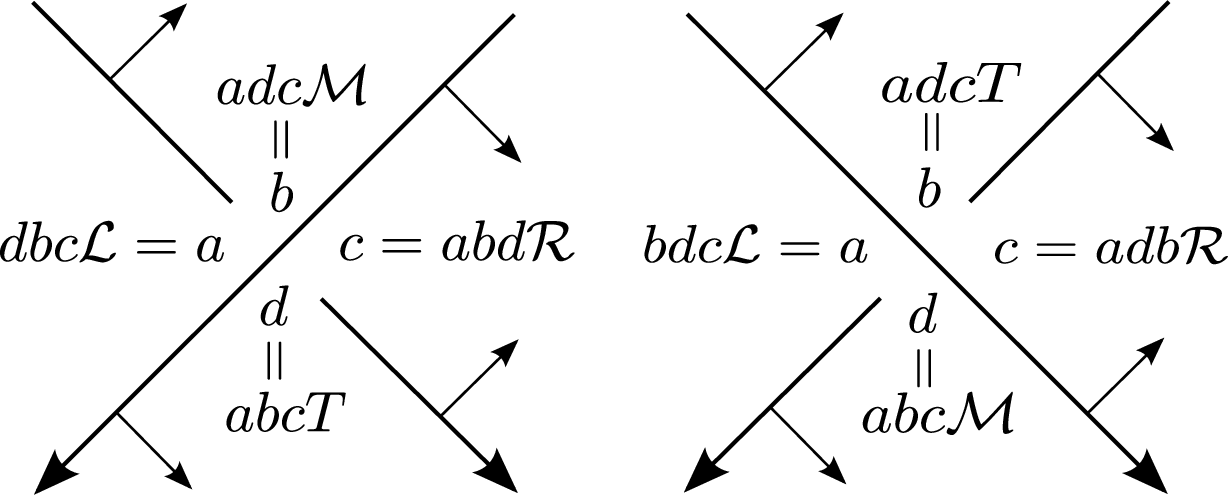}
\caption{Coloring classical crossings with a KTQ.}\label{quasio}
\end{center}
\end{figure}

In the reminder of the paper, $F$ will denote a compact oriented surface that may have a boundary.

\begin{definition} \label{classicalcol}
Let $D$ be a link diagram in the interior of $F$, or on the plane, and
let $(X,T)$ be a finite KTQ. {\it Regions} of $D$ are the connected components of $F\,\setminus$ universe of $D$ (or $\mathbb{R}^2 \setminus$ universe of $D$), and their set is denoted by $Reg$. A {\it KTQ-coloring} of $D$ is an assignment of elements of $X$ to the regions of $D$ satisfying the rule from Fig. \ref{quasio} at every crossing.
\end{definition}

\begin{remark}
For any choice of orientation, the coloring instructions for a KTQ $(X,T,\eL=T^{(2,3)},M=\hat{T},\R=T^{(1,2)})$ in Fig. \ref{Dehngen}(B) agree with the way of coloring described in Fig. \ref{quasio}. 
\end{remark}

We will briefly justify the fact of using ternary quasigroups.
In an oriented link diagram, each of the four corners around a crossing can be uniquely identified (for example, we can point to a corner adjacent to the two incoming edges). Fig. \ref{R2gen} represents a schematic colored second Reidemeister move, without specifying the types of crossings. Depending on the orientation, the corner colored by $x$ could be of any of the four types. If we are to have a coloring, then $x$ must exist for any $a$, $b$ and $c$. If the number of colorings of a diagram is to be unchanged by the move, then $x$ has to be unique. Thus, we reach a definition of a ternary quasigroup. We use its primary operation $T$ to color the corner adjacent to the outgoing edges of a positive crossing. In a negative crossing, $T$ is used for the corner adjacent to the incoming edges (Fig. \ref{quasio}).

\begin{lemma} \label{classicalRad}
For $D$ and $(X,T)$ as in Def. \ref{classicalcol}, the number of $KTQ$-colorings of $D$ does not change under Reidemeister moves. 
\end{lemma}
\begin{proof}
Let KTQ($D$) denote the abstract KTQ whose generators correspond to the regions in $Reg$, and relations correspond to crossings, and are as in Fig. \ref{quasio} (one of the four equivalent relations is assigned to each crossing). KTQ-colorings of $D$ can be understood as homomorphisms from KTQ($D$) to $(X,T)$. Therefore, to show that their number does not change under the Reidemeister moves, it is enough to show that the isomorphism class of KTQ($D$) does not change under these moves. 
KTQs can be defined equationally, thus they form a variety, and can be analyzed (up to isomorphism) using presentations and Tietze operations (see \cite{Cohn,Nie14}).
We leave the details of applying these operations to the reader.
\end{proof}

Consider links in the interior of $F\times I$, where $I$ denotes the interval. Then we can project links onto $F$ and work with the diagrams of links.
We have the following theorem as in \cite{Prz99} interpreting \cite{Hud69}.

\begin{theorem}\cite{Prz99,Hud69} \label{PrzHud}
Two link diagrams $D_1$ and $D_2$ in the interior of $F$ represent the same link in $F\times I$ if and only if one can go from $D_1$ to $D_2$ using Reidemeister moves and isotopy of $F$.
\end{theorem}

Thus, from Lemma \ref{classicalRad} and Theorem \ref{PrzHud} follows that KTQs yield invariants of links in $F\times I$.

\begin{figure}
\begin{center}
\includegraphics[height=2 cm]{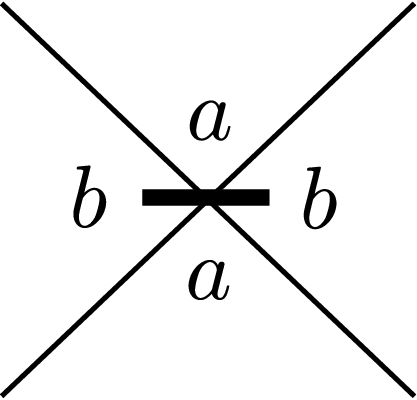}
\caption{The rule for assigning colors near Yoshikawa markers.}\label{yoshisurf}
\end{center}
\end{figure}

Knotted surfaces in $\mathbb{R}^4$ can be studied in various ways, e.g. using Yoshikawa diagrams \cite{Yo94,KJL15}. We can use KTQ-colorings for (oriented) Yoshikawa diagrams on $F$, or on the plane. Fig. \ref{yoshisurf} shows how to color around markers, regardless of the chosen orientation. Opposite corners are assigned the same color, and for classical crossings we use the convention from Fig. \ref{quasio}. A special case of such colorings for classical Yoshikawa diagrams was investigated in \cite{KN18}. 

One can assign an abstract $KTQ(D)$ to a given Yoshikawa diagram $D$, with generators corresponding to the regions of the diagram, and relations assigned to classical crossings as in Lemma \ref{classicalRad}. For a crossing with a marker, there are two relations equating the generators in the opposite corners of the crossing.

\begin{figure}
\begin{center}
\includegraphics[height=7.5 cm]{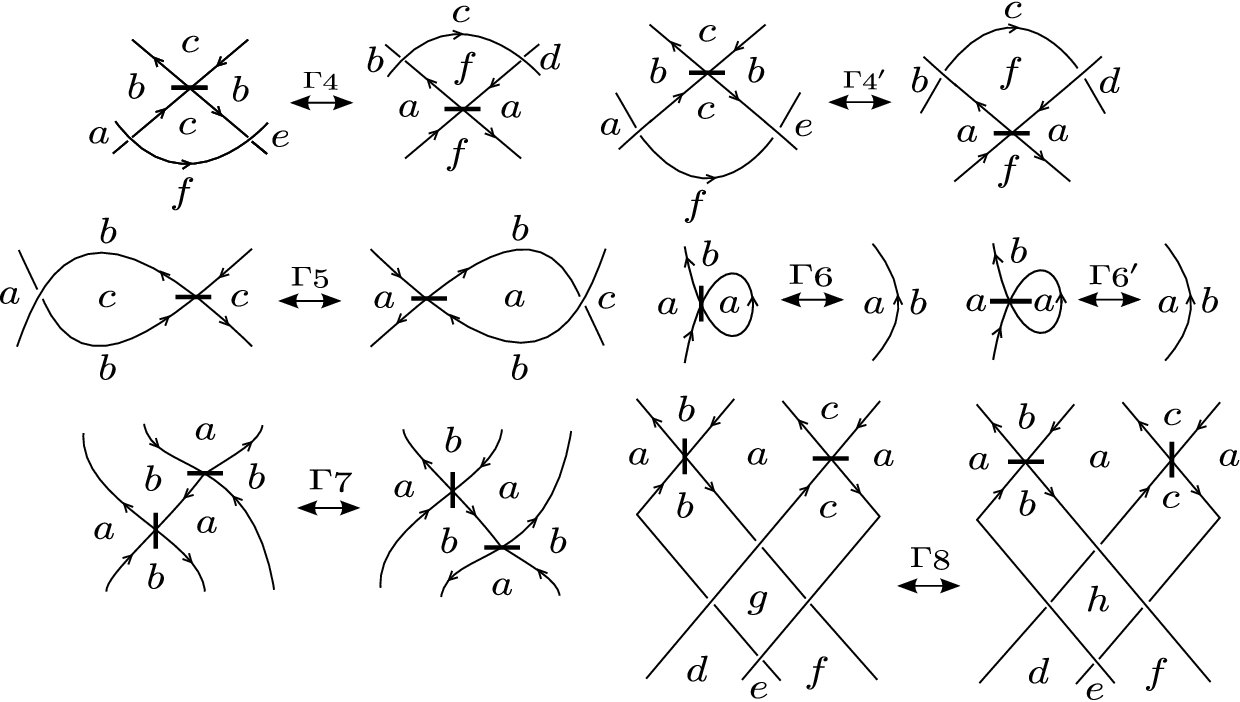}
\caption{A generating set of Yoshikawa moves.}\label{ymoves}
\end{center}
\end{figure}

\begin{lemma}\label{yoshiinv}
Let $D$ denote a Yoshikawa diagram on $F$ or on the plane. Then the isomorphism class of $KTQ(D)$ does not change under Yoshikawa moves. As a consequence, for a given KTQ, the number of KTQ-colorings is an invariant of Yoshikawa moves.
\end{lemma}
We leave the proof to the reader, noting that Fig. \ref{ymoves} shows the (KTQ-labeled) moves that together with the Reidemeister moves form a generating set for oriented Yoshikawa moves \cite{KJL15}.

Now we turn our attention to knotted surfaces described via projections to $\mathbb{R}^3$.

\begin{definition}
Let $S$ be a closed surface embedded smoothly in $\mathbb{R}^{4}$ and let $p\colon \mathbb{R}^{4}=\mathbb{R}^{3}\times\mathbb{R}\to\mathbb{R}^{3}$ be the projection.
$S$ is assumed to be in general position with respect to $p$.
By a {\it diagram} of $S$, we mean the image $p(S)$ equipped with the under-over information at each transverse double point. For example, one can use a {\it broken diagram} in which fragments of the projection are removed to indicate which part of the surface was higher before the projection. See \cite{KnSurf} for details on broken surface diagrams, and more general information about knotted surfaces and their descriptions. The closure of double point set of $p(S)$ is a graph with vertices of degree 1 (if there is a branch point) or 6 (for a triple point); loops with no vertices are also possible. We will refer to the edges of this graph (and to loops without vertices) as {\it double point edges}. 
\end{definition}
 
This time, the regions $Reg$ to which the elements of KTQs will be assigned are three-dimensional; they are the components of $\mathbb{R}^{3}-p(S)$. When working with KTQ invariants of surfaces, it is convenient to have co-orientation.

\begin{definition}{\cite{KnSurf, CJKLS03}}
Suppose that the surface $S$ is oriented. We can give a co-orientation to the complement of the branch point set $Br$ on $p(S)$ as follows. For a point 
$x\in p(S)- Br$, choose vectors $v_1$, $v_2$ that are tangent to $p(S)$ in $\mathbb{R}^3$, so that the oriented frame $(v_1,v_2)$ matches the orientation of $S$. Then a normal vector $v(x)$ in $\mathbb{R}^3$ is chosen so that the ordered triple $(v,v_1,v_2)$ matches the orientation of 3-space.
\end{definition}

Locally, there are four regions near a double point edge. The following definitions will simplify the description of KTQ-colorings and homological considerations.

\begin{definition}
For a double point edge in a projection $p(S)$ (resp. double point in a classical knot diagram), the {\it source region} is the region such that all co-orientation arrows point out of it, and the {\it target region} is the region such that all co-orientation arrows point into it. An {\it ascending path} is a triple of regions $(r_1,r_2,r_3)$, where $r_1$ is the source region, $r_1$ and $r_2$ are separated by an under-sheet (resp. under-arc), and $r_3$ is the target region. If elements of a given KTQ $(X,T)$ are assigned to the regions via a function $c\colon Reg\to X$, then a {\it colored path} is the triple 
$(c(r_1),c(r_2),c(r_3))$.
\end{definition}

\begin{figure}
\begin{center}
\includegraphics[height=4 cm]{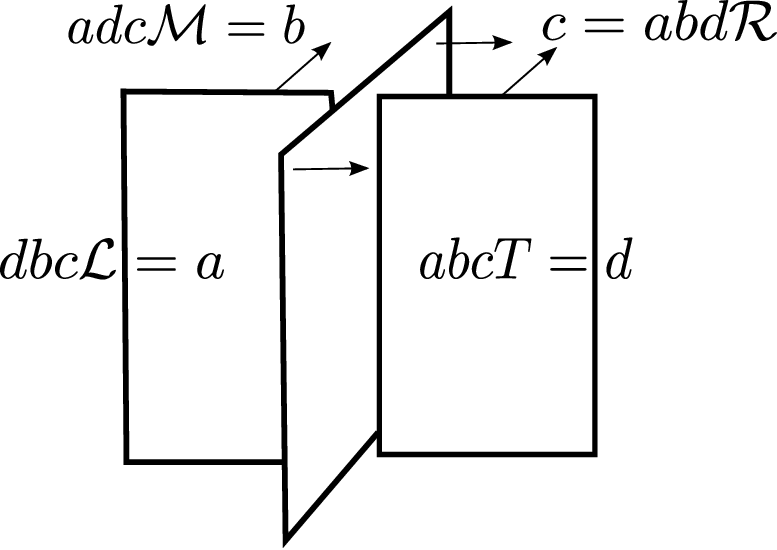}
\caption{The coloring rule around an edge of double points in a surface diagram.}\label{surfcr}
\end{center}
\end{figure}

\begin{definition}
Let $(X,T)$ be a KTQ, and let $D$ denote an oriented knotted surface diagram. A {\it KTQ-coloring} is a function $c\colon Reg\to X$ such that for each double point edge $\gamma$ of $D$, the region near $\gamma$ which is not in its ascending path 
$(r_1,r_2,r_3)$, receives the color $c(r_1)c(r_2)c(r_3)T$, see Fig. \ref{surfcr}. 
Note that the coloring rule in Fig. \ref{quasio} can be understood similarly.
\end{definition}

\begin{definition}
For an oriented knotted surface diagram $D$, we can assign to it an abstract knot-theoretic ternary quasigroup $KTQ(D)$, with generators in one to one correspondence with regions, and relations corresponding to double point edges. For a double point edge $\gamma$ with its ascending path $(r_1,r_2,r_3)$ and its fourth region $r_4$, the relation is $r_4=r_1r_2r_3T$.
\end{definition}

\begin{figure}
\begin{center}
\includegraphics[height=7 cm]{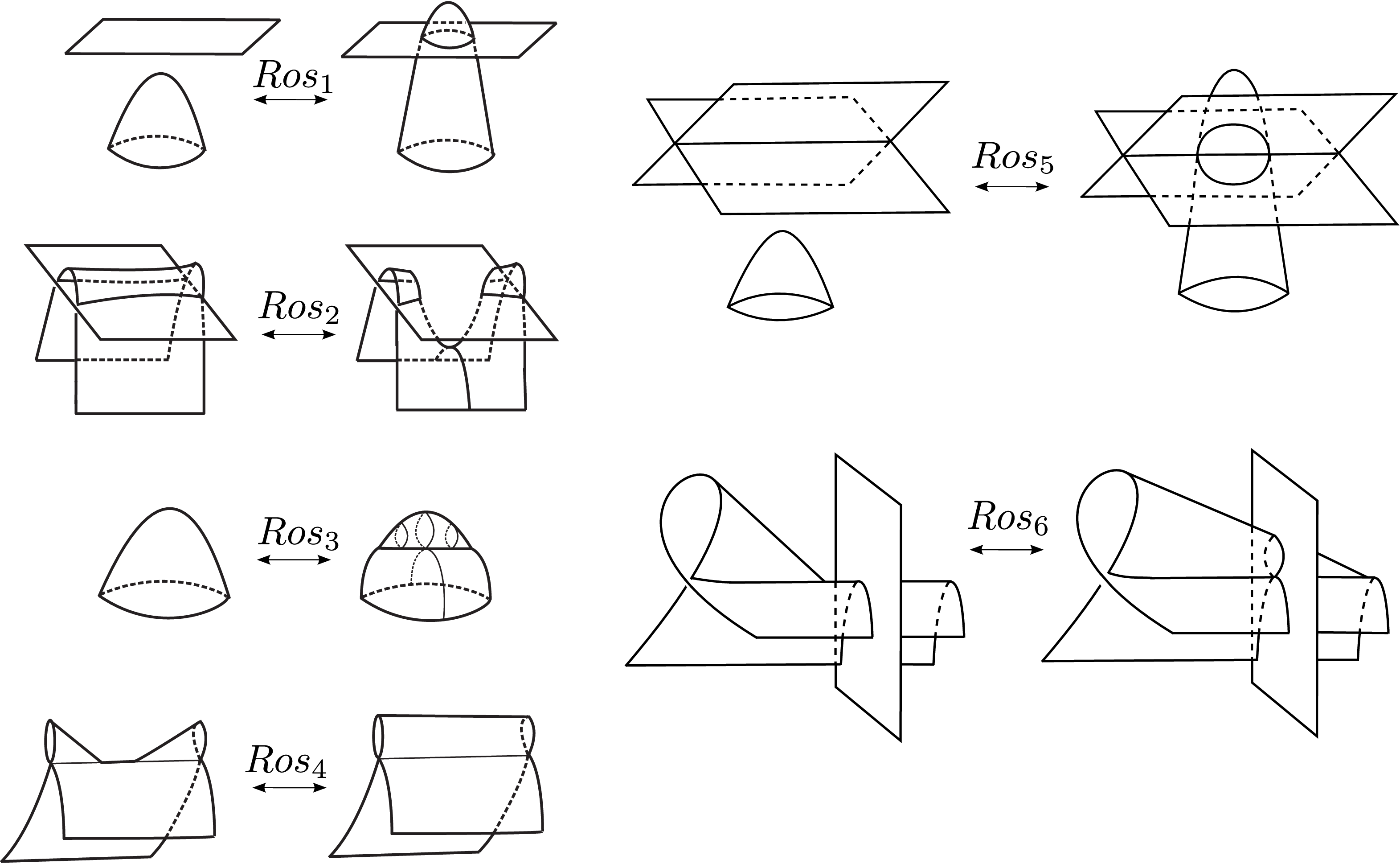}
\caption{Roseman moves 1-6.}\label{Ros}
\end{center}
\end{figure}

It is known that two knotted surface diagrams $D_1$ and $D_2$ represent the same knotted surface if and only if they are related by a finite sequence of Roseman moves \cite{Ros95} (Fig. \ref{Ros} and Fig. \ref{Ros7}). 

\begin{figure}
\begin{center}
\includegraphics[height=6.9 cm]{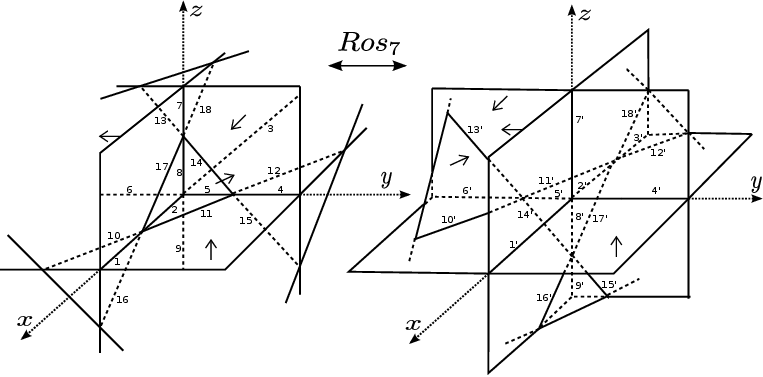}
\caption{The seventh Roseman move with an assigned orientation and relation labels.}\label{Ros7}
\end{center}
\end{figure}

\begin{theorem} \label{Rosthm}
Roseman moves do not change the isomorphism class of $KTQ(D)$. Thus, the number of KTQ-colorings of a knotted surface diagram is an invariant of the surface.
\end{theorem}

\begin{proof}
We prove the invariance under the seventh Roseman move, leaving the remaining moves to the reader.
In the seventh Roseman move, Fig. \ref{Ros7}, there is a triple point in the intersection of three planes (in our illustration they are the $xy$-plane, $xz$-plane and the $yz$-plane, and the triple point is (0,0,0)). The fourth plane (say, $x+y+z=1$) moves to the other side of the triple point. This move is also called tetrahedral move, as it involves a tetrahedron in the first octant (depicted with solid lines) before the move, and another one after the move (in the (-,-,-) octant). The co-orientation is as shown in the figure, and we assume that the sheets are ordered from the highest (unbroken in broken diagrams) to the lowest (the most broken) as follows:
the $xy$-plane, the $xz$-plane, the $yz$-plane, and the $x+y+z=1$ plane.
The symbols of generators are organized according to octants. All the octants, except one, are divided by the lowest plane (both before and after the move), and  they are identified by the signs of coordinates.
The symbols of generators are with bars if they are behind the lowest plane, and without bars if they are in front of it. The following is the list of octants with generators assigned to their regions: (+,+,+): $a$, $\bar{a}$; (-,+,+): $b$, $\bar{b}$; (-,-,+): $c$, $\bar{c}$; (+,-,+): $d$, $\bar{d}$;
(+,+,-): $e$, $\bar{e}$; (-,+,-): $f$, $\bar{f}$; (-,-,-): $g$, $\bar{g}$;
(+,-,-): $h$, $\bar{h}$. Before the move, the relations assigned to double point edges (with their numbers shown in Fig. \ref{Ros7}, but see also Fig. \ref{triplecrs}) are as follows:
(1) $a=ehdT$, (2) $\bar{a}=\bar{e}\bar{h}\bar{d}T$, (3) $\bar{b}=\bar{f}\bar{g}\bar{c}T$, (4) $b=feaT$, (5) $\bar{b}=\bar{f}\bar{e}\bar{a}T$, (6) $\bar{c}=\bar{g}\bar{h}\bar{d}T$,
(7) $c=badT$, (8) $\bar{c}=\bar{b}\bar{a}\bar{d}T$, (9) $\bar{g}=\bar{f}\bar{e}\bar{h}T$, (10) $\bar{d}=\bar{h}hdT$, (11) $\bar{a}=\bar{e}eaT$, (12) $\bar{b}=\bar{f}fbT$, (13) $\bar{d}=\bar{c}cdT$, (14)
$\bar{a}=\bar{b}baT$, (15) $\bar{e}=\bar{f}feT$, (16) $\bar{h}=\bar{e}ehT$,
(17) $\bar{d}=\bar{a}adT$, and (18) $\bar{c}=\bar{b}bcT$. After the move, the relations are: (1') $a=ehdT$, (2') $b=fgcT$, (3') $\bar{b}=\bar{f}\bar{g}\bar{c}T$,
(4') $b=feaT$, (5') $c=ghdT$, (6') $\bar{c}=\bar{g}\bar{h}\bar{d}T$,
(7') $c=badT$, (8') $g=fehT$, (9') $\bar{g}=\bar{f}\bar{e}\bar{h}T$, (10') $\bar{d}=\bar{h}hdT$, (11') $\bar{c}=\bar{g}gcT$, (12') $\bar{b}=\bar{f}fbT$,
(13') $\bar{d}=\bar{c}cdT$, (14') $\bar{h}=\bar{g}ghT$, (15') $\bar{e}=\bar{f}feT$,
(16') $\bar{h}=\bar{e}ehT$, (17') $\bar{g}=\bar{f}fgT$, and (18') $\bar{c}=\bar{b}bcT$. The intersection of these two sets of relations does not contain (2), (5), (8), (11), (14), (17), and (2'), (5'), (8'), (11'), (14'), (17'). We will show that in these two six-element subsets, five relations are the consequences of the other relations from their sets, and the last one can be removed together with the generator corresponding to the tetrahedral region. 
Thus, we will obtain identical presentations before and after the move.
We will reveal relations as consequences, and remove them, in the following order:
(14), (17), (11), (8), (5):
\begin{align*}
&\bar{b}baT\stackrel{(12)}{=}(\bar{f}fbT)baT\stackrel{(4)}{=}
[\bar{f}f(feaT)T](feaT)aT\stackrel{(A1)}{=}(\bar{f}feT)eaT\stackrel{(15)}{=}
\bar{e}eaT\stackrel{(11)}{=}\bar{a},\\
&\bar{a}adT\stackrel{(11)}{=}(\bar{e}eaT)adT\stackrel{(1)}{=}
[\bar{e}e(ehdT)T](ehdT)dT\stackrel{(A1)}{=}(\bar{e}ehT)hdT\stackrel{(16)}{=}
\bar{h}hdT\stackrel{(10)}{=}\bar{d},\\
&\bar{e}eaT\stackrel{(1)}{=}\bar{e}e(ehdT)T\stackrel{(A2)}{=}
\bar{e}(\bar{e}ehT)[(\bar{e}ehT)hdT]T\stackrel{(16)}{=}\bar{e}\bar{h}(\bar{h}hdT)T
\stackrel{(10)}{=}\bar{e}\bar{h}\bar{d}T\stackrel{(2)}{=}\bar{a},\\
&\bar{b}\bar{a}\bar{d}T\stackrel{(5)}{=}(\bar{f}\bar{e}\bar{a}T)\bar{a}\bar{d}T\stackrel{(2)}{=}
[\bar{f}\bar{e}(\bar{e}\bar{h}\bar{d}T)T](\bar{e}\bar{h}\bar{d}T)\bar{d}T
\stackrel{(A1)}{=}(\bar{f}\bar{e}\bar{h}T)\bar{h}\bar{d}T\stackrel{(9)}{=}
\bar{g}\bar{h}\bar{d}T\stackrel{(6)}{=}\bar{c},\\
&\bar{f}\bar{e}\bar{a}T\stackrel{(2)}{=}
\bar{f}\bar{e}(\bar{e}\bar{h}\bar{d}T)T\stackrel{(A2)}{=}
\bar{f}(\bar{f}\bar{e}\bar{h}T)[(\bar{f}\bar{e}\bar{h}T)\bar{h}\bar{d}T]T
\stackrel{(9)}{=}\bar{f}\bar{g}(\bar{g}\bar{h}\bar{d}T)T\stackrel{(6)}{=}
\bar{f}\bar{g}\bar{c}T\stackrel{(3)}{=}\bar{b}.
\end{align*}
Now the generator $\bar{a}$ appears only in (2), so we remove it together with its relation.
For the second subset of relations, we use the order
(14'), (11'), (17'), (2'), (5'):
\begin{align*}
&\bar{g}ghT\stackrel{(17')}{=}(\bar{f}fgT)ghT\stackrel{(8')}{=}
[\bar{f}f(fehT)T](fehT)hT\stackrel{(A1)}{=}
(\bar{f}feT)ehT\stackrel{(15')}{=}\bar{e}ehT\stackrel{(16')}{=}\bar{h},\\
&\bar{g}gcT\stackrel{(17')}{=}(\bar{f}fgT)gcT\stackrel{(A1)}{=}
[\bar{f}f(fgcT)T](fgcT)cT\stackrel{(2')}{=}(\bar{f}fbT)bcT\stackrel{(12')}{=}
\bar{b}bcT\stackrel{(18')}{=}c,\\
&\bar{f}fgT\stackrel{(8')}{=}\bar{f}f(fehT)T\stackrel{(A2)}{=}
\bar{f}(\bar{f}feT)[(\bar{f}feT)ehT]T\stackrel{(15')}{=}
\bar{f}\bar{e}(\bar{e}ehT)T\stackrel{(16')}{=}\bar{f}\bar{e}\bar{h}T
\stackrel{(9')}{=}\bar{g},\\
&fgcT\stackrel{(5')}{=}fg(ghdT)T\stackrel{(8')}{=}
f(fehT)[(fehT)hdT]T\stackrel{(A2)}{=}fe(ehdT)T\stackrel{(1')}{=}feaT\stackrel{(4')}{=}b,\\
&ghdT\stackrel{(8')}{=}(fehT)hdT\stackrel{(A1)}{=}
[fe(ehdT)T](ehdT)dT\stackrel{(1')}{=}(feaT)adT\stackrel{(4')}{=}badT\stackrel{(7')}{=}c.
\end{align*}
The generator $g$ is now only present in the relation (8'), thus we remove it with this relation, obtaining identical presentations before and after the move.
As explained in \cite{CJKLS03}, other versions of the seventh Roseman move can be obtained using this particular version and the fifth Roseman moves. 
\end{proof}
In the next section, we will use KTQ-colorings as a basis for (co)homological invariants.

\section{Geometric applications of KTQ homology}
\subsection{KTQ homology and the cycles associated with diagrams}

\begin{definition}
Let $(X,T)$ be a KTQ.
With the notation as in Def. \ref{maindefs}, Thm. \ref{mainhom}, and Def. \ref{maindeg}, we define KTQ homology, $H^N(X,T)$, as homology of the quotient complex
\[ 
(C_n^{N}(X,T),\partial_n):=(C_n(X)/C_n^D(X,T),\partial_n^{(1,-1)}).
\]
\end{definition}

\begin{example} In low dimensions the differential 
$\partial_n=\partial_n^L-\partial_n^R$ is as follows:
\[
\partial_0(a,b)=b-a,
\]
\begin{align*}
\partial_1(a,b,c)&=(b,c)-(a,abcT)\\
 &-(abcT,c)+(a,b),
\end{align*}
\begin{align*} 
\partial_2(a,b,c,d)&=(b,c,d)-(a,abcT,(abcT)cdT)\\
&-(abcT,c,d)+(a,b,bcdT)\\
&+(ab(bcdT)T,bcdT,d)-(a,b,c),
\end{align*}
\begin{align*}
\partial_3(a,b,c,d,e)&=(b,c,d,e)-(a,abcT,(abcT)cdT,[(abcT)cdT]deT)\\
&-(abcT,c,d,e)+(a,b,bcdT,(bcdT)deT)\\
&+(ab(bcdT)T,bcdT,d,e)-(a,b,c,cdeT)\\
&-(ab[bc(cdeT)T]T,bc(cdeT)T,cdeT,e)+(a,b,c,d).
\end{align*}
\end{example}

\begin{remark}
Given an $n$-element KTQ $(X,T)$ that colors a knot diagram, one can consider an associated $n^2$-element structure $LB(X,T)$ consisting of pairs of elements of $X$, with an idea that a pair $(a,b)$ is assigned to an arc with neighbouring regions colored by $a$ and $b$. The KTQ operation then induces $2n$ partial binary operations on $X^2$, defined when two suitable coordinates in the two pairs are the same. Such a structure was defined in \cite{NOO19}, and the authors showed that the (co)homology groups that they construct for it are isomorphic to KTQ (co)homology groups. The differential that was obtained has a somewhat similar form to biquandle homology differential, but it involves aforementioned $2n$ operations, and therefore also a ternary operation and pairs of elements. The authors also noticed the 
equivalence of definitions of degenerate triples for ternary quasigroups: they can be defined as triples $a$, $b$, $c$ with $abcT=b$, or by $a$, $b$, $abb\R$, for any $a$ and $b\in X$. Indeed, $b=abcT$ if and only if $c=abb\R$.
\end{remark}

Let $D$ denote either a link diagram or a Yoshikawa diagram (on a compact oriented surface $F$, or on a plane), or a knotted surface diagram in $\mathbb{R}^3$.
If $D$ is colored by the elements of a given KTQ $(X,T)$, then we can assign to it a cycle with respect to the differential $\partial_n$. We will show that the homology class of this cycle is not changed by the moves of the type applicable to the diagram $D$ (Reidemeister, Yoshikawa, or Roseman moves). 

\begin{definition}{\cite{KnSurf, CJKLS03}}
The sign of a triple point in a surface diagram is defined as follows. Let $v_T$, $v_M$ and $v_B$ be co-orientation normal vectors to the top, middle, and bottom sheets, respectively, that intersect at a triple point. If the oriented frame $(v_T,v_M,v_B)$ coincides with the right hand
orientation of 3-space, the triple point is said to be {\it positive}, otherwise it is {\it negative}.

For a triple point, the {\it source region} is the region near the triple point such that $v_T$, $v_M$ and $v_B$ point away from it, and the {\it target region} is the region such that $v_T$, $v_M$ and $v_B$ point towards it. 
\end{definition}  

\begin{definition}
For a triple point of an oriented surface diagram, the term {\it ascending path} will mean a sequence of regions $(r_0,r_1,r_2,r_3)$, starting with the source region, and ending with the target region, such that the regions $r_0$ and $r_1$ are separated by the bottom sheet, $r_1$ and $r_2$ by the middle sheet, and $r_2$ and $r_3$ by the top sheet. Let $c\colon Reg\to X$ be a function from the set of regions of the surface diagram to a fixed KTQ $(X,T)$. Then the sequence
$(c(r_0),c(r_1),c(r_2),c(r_3))$ will be called a {\it colored path}.
\end{definition}

\begin{definition} \label{associatedchain}
Let $D$ be a KTQ-colored link diagram or a Yoshikawa diagram on a compact oriented surface $F$, or on a plane. We define its {\it associated chain}, $c_D$, as the sum of the colored paths over all crossings of $D$, taken with the sign of a crossing. If $D$ denotes a KTQ-colored knotted surface diagram in $\mathbb{R}^3$, then the associated chain $c_D$ is the sum of the colored paths over all triple points of $D$, taken with the sign of a triple point.
\end{definition}

\begin{figure}
\begin{center}
\includegraphics[height=4 cm]{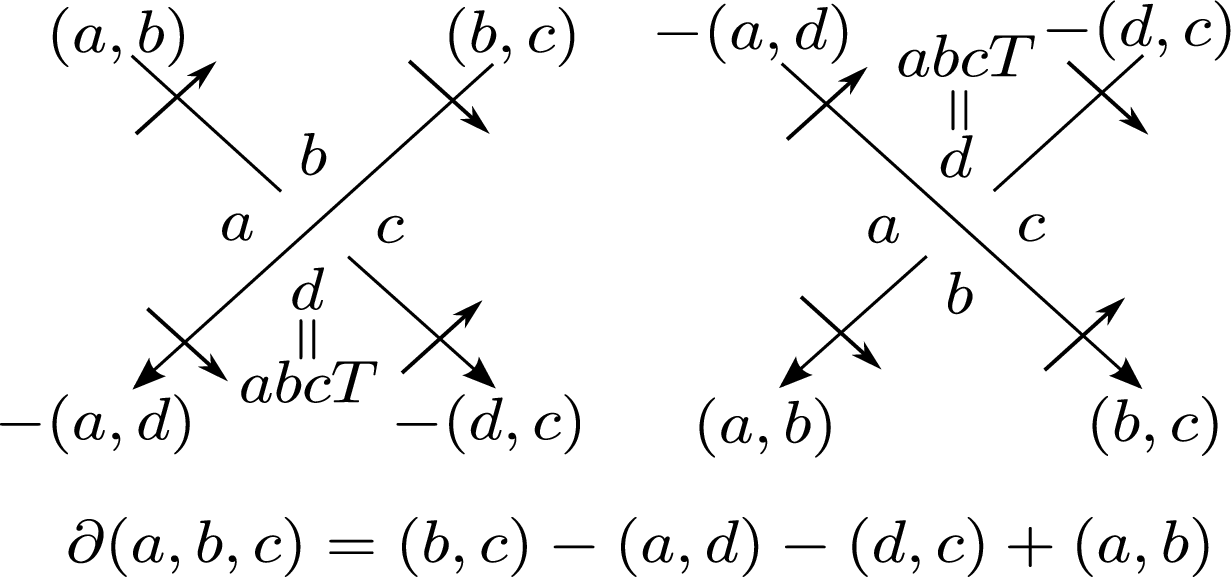}
\caption{The differential of a colored path of a crossing.}\label{quasidiff}
\end{center}
\end{figure}

\begin{figure}
\begin{center}
\includegraphics[height=3 cm]{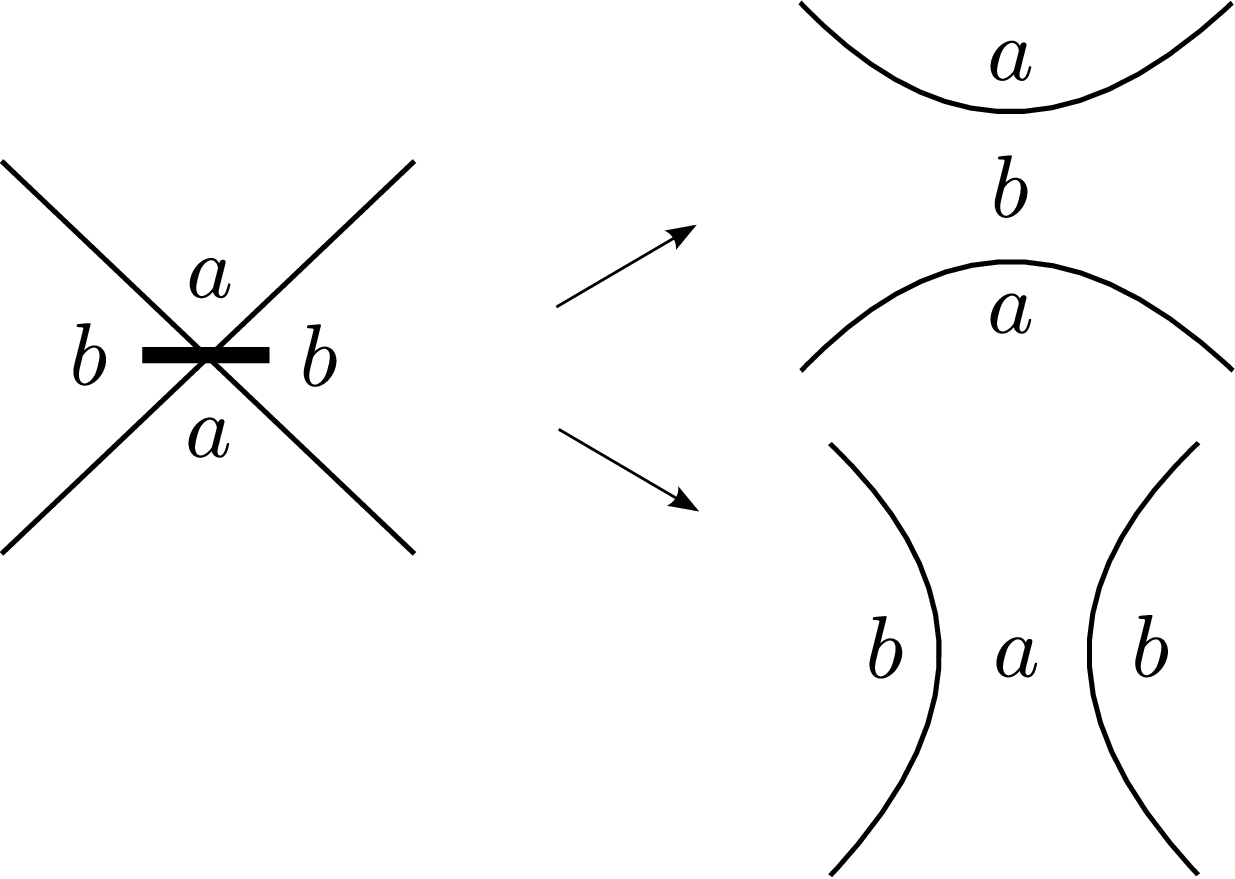}
\caption{Smoothings of Yoshikawa markers are compatible with KTQ-colorings.}\label{yoshism}
\end{center}
\end{figure}

\begin{lemma} For a KTQ-colored oriented link or Yoshikawa diagram $D$ on a compact oriented surface $F$ (or a plane), its associated chain $c_D$ is a cycle with respect to the differential $\partial_1$.
\end{lemma} 
\begin{proof}
First, let $D$ be a link diagram.
We see from Fig. \ref{quasidiff}, that in the differential 
\begin{align*}
\partial_1(a,b,c)&= (b,c)-(a,abcT)-(abcT,c)+(a,b)\\
&=(b,c)-(a,d)-(d,c)+(a,b)
\end{align*}
of a colored path of a positive crossing, the positive pairs of colors correspond to the incoming edges, and the negative pairs can be assigned to the outgoing edges. For a negative crossing, the chain assigned to it is $-(a,b,c)$. Thus,
in $\partial_1[-(a,b,c)]=-(b,c)+(a,d)+(d,c)-(a,b)$, again the positive pairs correspond to the incoming edges, and the negative pairs to the outgoing edges. Also, the order of colors in each pair is pointed by the co-orientation of a diagram.
This ensures that the two ends of an edge give the same pair of colors, but with opposite signs. Thus, $\partial_1(c_D)=0$. 
If $D$ is a Yoshikawa diagram, then smoothing all the markers in one of the ways shown in Fig. \ref{yoshism} is compatible with the KTQ-coloring. Then it is enough to notice that $c_D$ is the same as the associated cycle of a link diagram obtained after marker smoothings.
\end{proof}

\begin{figure}
\begin{center}
\includegraphics[height=6cm]{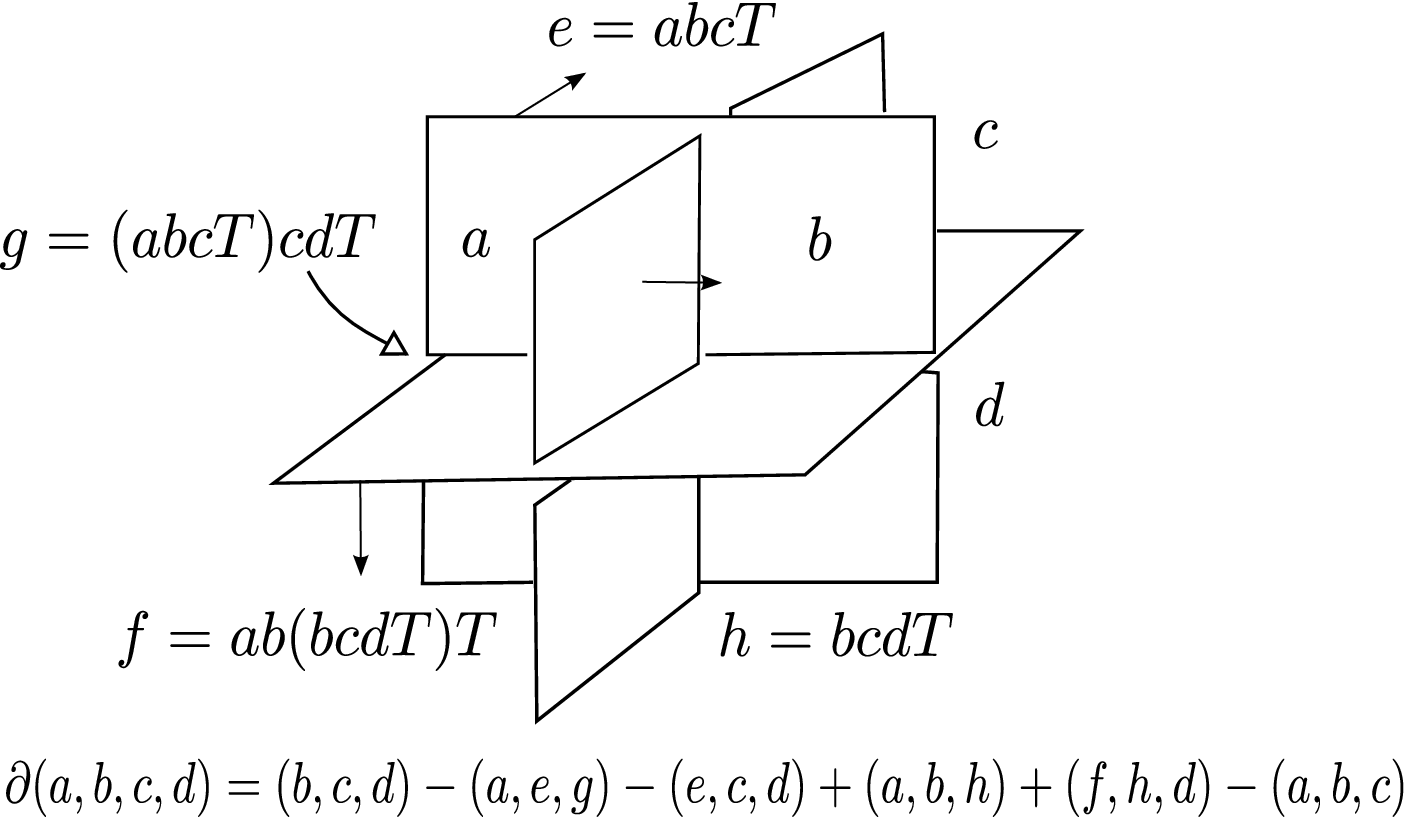}
\caption{A triple point with the differential of a colored path. }\label{triplecrdiff}
\end{center}
\end{figure}

\begin{definition}[\cite{KnSurf}]
For an oriented surface $S$, each double point edge $\gamma$ is oriented as follows: if $v_t$ is the oriented normal to the top sheet, 
$v_b$ is the oriented normal to the bottom sheet, and $v_\gamma$ is the vector tangent to the edge $\gamma$, then it is required that the ordered triple $(v_\gamma,v_t,v_b)$ matches the orientation of $3$-space by the right-hand convention. A triple point always has three double point edges oriented towards it, and three edges with orientation out of it. 
\end{definition}

\begin{lemma}
For a KTQ-colored knotted surface diagram $D$ in $\mathbb{R}^3$, its associated chain $c_D$ is a cycle with respect to the differential $\partial_2$.
\end{lemma}

\begin{proof}
Calculating the differential on a colored path of a triple point corresponds to taking a sum of six suitably signed colored paths of double point edges connected at the triple point.
The incoming double point edges are assigned the colored paths with a minus sign, the outgoing edges receive positive colored paths, as in Fig. \ref{triplecrdiff}. Thus, if a double point edge ends with triple points, the signed triples assigned to it at its ends will cancel out. If an edge ends in a branch point, then the triple assigned to it is a degenerate cycle of the form $\pm (x,y,xyy\R)$.
\end{proof}

\begin{figure}
\begin{center}
\includegraphics[height=2.5 cm]{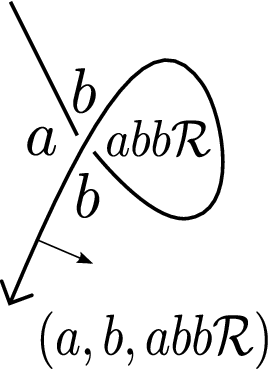}
\caption{The degenerate cycle for the first Reidemeister move.}\label{degR1os}
\end{center}
\end{figure}

\begin{lemma} \label{Reidhominv}
The KTQ homology class of a cycle $c_D$ assigned to an oriented and colored link diagram is invariant under Reidemeister moves.
\end{lemma}

\begin{proof}
The first Reidemeister move adds or removes a degenerate cycle $(a,b,abb\R)$ (Fig. \ref{degR1os}), so it does not change the KTQ homology class. The contributions coming from the two crossings in the second Reidemeister move cancel out, because the crossings have opposite signs. Now consider the third Reidemeister move with all the crossings positive, with labels as in Fig. \ref{scattering}, with orientation from top to bottom. The contributions from the crossings after the move, minus the contributions before the move are equal to the boundary 
\begin{align*}
\partial_2(a,b,c,d) & =(b,c,d)-(a,abcT,(abcT)cdT)\\
& -(abcT,c,d)+(a,b,bcdT)\\
& +(ab(bcdT)T,bcdT,d)-(a,b,c).
\end{align*}
\end{proof}

\begin{lemma}
The KTQ homology class of a cycle $c_D$ assigned to an oriented and KTQ-colored Yoshikawa diagram $D$ is invariant under Yoshikawa moves. If the diagram is on a plane, then $c_D$ represents zero in homology.
\end{lemma}

\begin{proof}
With the invariance under the Reidemeister moves proved as in Lemma \ref{Reidhominv},
we need to check the moves with markers from Fig. \ref{ymoves}. In the colored moves $\Gamma_4$ and $\Gamma_4'$, it is easy to see that $a=e$ and $b=d$. Thus, both before and after the moves, two classical crossings with the opposite signs yield the same colored paths, and the total contribution is zero.
In colored $\Gamma_5$, both sides of the move contribute $(a,b,c)$. $\Gamma_6$, $\Gamma_6'$, and $\Gamma_7$ do not have any classical crossings, so do not contribute any colored paths. In colored $\Gamma_8$, it can be shown that $d=c$, $a=e$, and $b=f$. The chain assigned to the left side of the figure is
$(g,c,a)-(g,c,a)+(g,c,a)-(g,c,a)=0$, and on the right side it is 
$(h,b,a)-(h,b,a)+(h,b,a)-(h,b,a)=0$.

The condition required of Yoshikawa diagrams on the plane is that if all the markers are smoothed in the same way (as in Fig. \ref{yoshism}), the resulting diagram is a diagram of an unlink. Since the cycle $c_D$ assigned to a Yoshikawa diagram is the same as the cycle assigned to the link obtained after its marker smoothings, in case of $D$ on the plane, $c_D$ will be homologically trivial.
\end{proof}

\begin{figure}
\begin{center}
\includegraphics[height=8.5 cm]{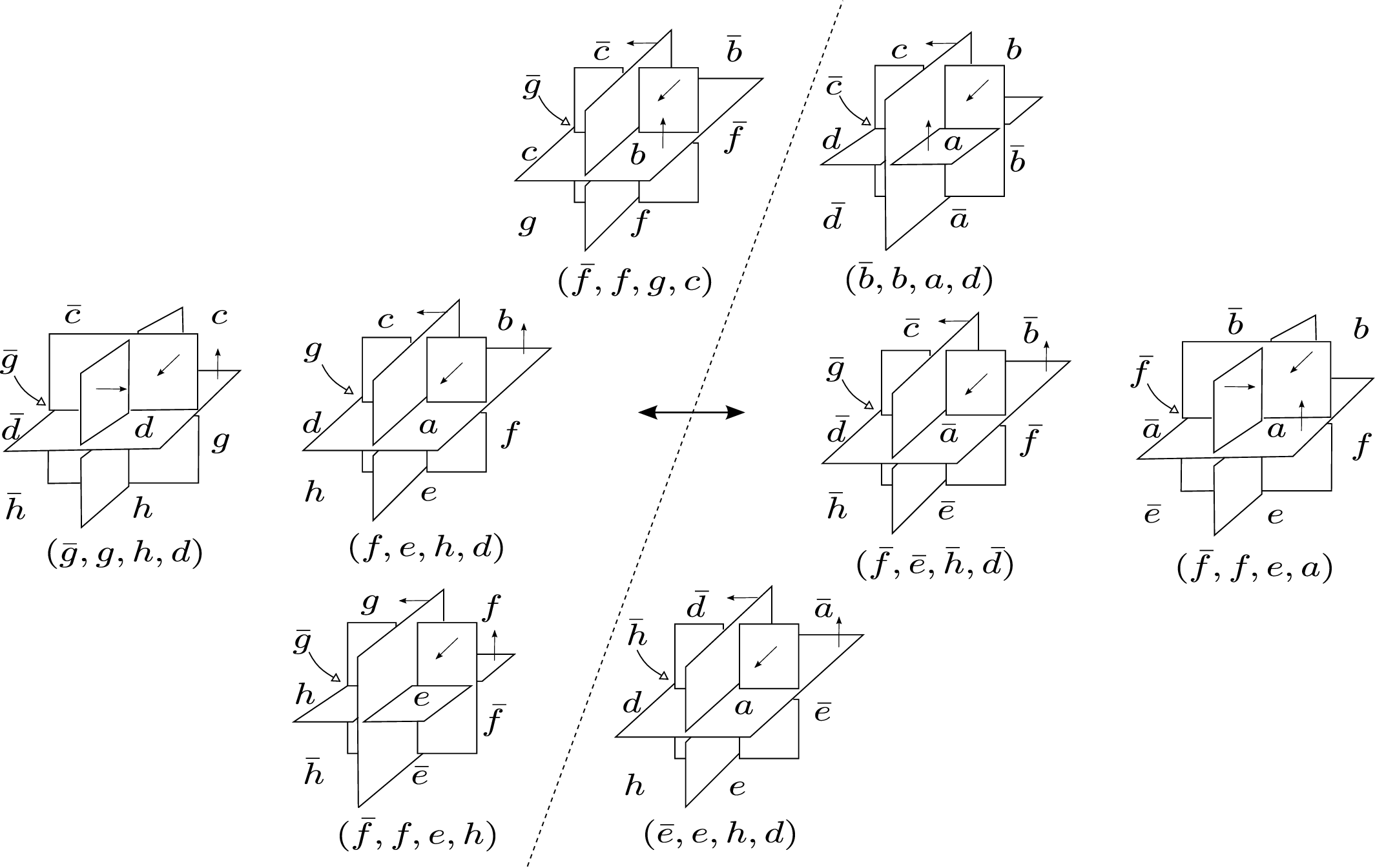}
\caption{The colored triple points involved in the seventh Roseman move, together with their colored paths.}\label{triplecrs}
\end{center}
\end{figure}

\begin{lemma}
The KTQ homology class of a cycle $c_D$ assigned to an oriented and colored broken surface diagram is an invariant under Roseman moves.
\end{lemma}
\begin{proof}
We only need to consider the Roseman moves that involve triple points (Fig. \ref{Ros} and Fig. \ref{Ros7}). The fifth Roseman move is similar to the second Reidemeister move in that the contributions from the two triple points with opposite signs cancel out.
The sixth Roseman move introduces (or deletes) a degenerate cycle. As for the seventh Roseman move (Fig. \ref{Ros7}), with labels and under-over information as in the proof of Theorem \ref{Rosthm}, 
the contributions from the triple points after the move, minus the contributions before the move, are equal to the boundary 
\begin{align*}
\partial_3(\bar{f},f,e,h,d)& =(f,e,h,d)-(\bar{f},\bar{f}feT,(\bar{f}feT)ehT,
[(\bar{f}feT)ehT]hdT)\\
& -(\bar{f}feT,e,h,d)+(\bar{f},f,fehT,(fehT)hdT)\\
& +(\bar{f}f(fehT)T,fehT,h,d)-(\bar{f},f,e,ehdT)\\
& -(\bar{f}f[fe(ehdT)T]T,fe(ehdT)T,ehdT,d)+(\bar{f},f,e,h)\\
&=(f,e,h,d)-(\bar{f},\bar{e},\bar{h},\bar{d})-(\bar{e},e,h,d)+(\bar{f},f,g,c)\\
&+(\bar{g},g,h,d)-(\bar{f},f,e,a)-(\bar{b},b,a,d)+(\bar{f},f,e,h).
\end{align*} 
The triple points in this configuration, together with the colors, co-orientation, under-over information, and colored paths, are depicted in Fig. \ref{triplecrs} (the left and right sides are switched compared to Fig. \ref{Ros7}). All the triple points are positive, thus the colored paths are taken with the positive signs.
\end{proof}

We define KTQ cohomology groups in a standard dual way, so the cocycles used for knot diagrams are functions 
$f\colon X\times X\times X\to A$, where $A$ is an abelian group and $(X,T)$ is a KTQ, satisfying
two conditions:
\begin{equation*}
\forall_{a,b \in X} \quad f(a,b,abb\R)=0,
\end{equation*}
\begin{align*}
\forall_{a,b,c,d \in X} \quad 
& f(b,c,d)-f(a,abcT,(abcT)cdT)\\
&-f(abcT,c,d)+f(a,b,bcdT)\\
&+f(ab(bcdT)T,bcdT,d)-f(a,b,c)=0.
\end{align*}
Cocycles for knotted surface diagrams (in $\mathbb{R}^3$) are 
functions 
$\phi\colon X\times X\times X\times X\to A$, satisfying:
\begin{equation*}
\forall_{a,b,c \in X} \quad \phi(a,b,abb\R,c)=\phi(c,a,b,abb\R)=0,
\end{equation*}
\begin{align*}
\forall_{a,b,c,d,e \in X} \quad
&\phi(b,c,d,e)-\phi(a,abcT,(abcT)cdT,[(abcT)cdT]deT)\\
&-\phi(abcT,c,d,e)+\phi(a,b,bcdT,(bcdT)deT)\\
&+\phi(ab(bcdT)T,bcdT,d,e)-\phi(a,b,c,cdeT)\\
&-\phi(ab[bc(cdeT)T]T,bc(cdeT)T,cdeT,e)+\phi(a,b,c,d)=0.
\end{align*}

In an analogous way to the construction in \cite{CJKLS03}, we define cocycle invariants.

\begin{definition}
Let $(X,T)$ be a KTQ. Let $f$ denote a cocycle from the KTQ cohomology of $(X,T)$, with inputs from $C_1(X)$ for link and Yoshikawa diagrams, and from $C_2(X)$ for knotted surface diagrams. $f$ takes values in an abelian group $A$ written multiplicatively.
Let $\mathcal{C}$ denote an $(X,T)$-coloring of a link or Yoshikawa diagram $D$ on a compact oriented surface $F$ (or on a plane), or a coloring of a knotted surface diagram $D$ in $\mathbb{R}^3$. For such $\mathcal{C}$, and a classical crossing or a triple point $\tau$, let $cp(\tau,\mathcal{C})$ denote the colored path of $\tau$, and let $c_{D,\mathcal{C}}$ be the cycle assigned to the colored diagram $D$.
Then we define the {\it cocycle invariant} as the state-sum expression
\[ 
\Phi(D,f)=\sum_{\mathcal{C}}\prod_{\tau}f(cp(\tau,\mathcal{C}))^{\epsilon(\tau)}=
\sum_{\mathcal{C}}f(c_{D,\mathcal{C}}),
\]
where the product is taken over all classical crossings or triple points of $D$, the sum is taken over all $(X,T)$-colorings of $D$, and $\epsilon(\tau)$ denotes the sign of $\tau$. The value of $\Phi(D,f)$ is in the group ring $\Z[A]$.
\end{definition}

\begin{lemma}
If $D$ denotes a link diagram (on a compact oriented surface $F$, or on a plane), then $\Phi(D,f)$ is not changed by Reidemeister moves. If $D$ is a Yoshikawa diagram
(on a compact oriented surface $F$, or on a plane), then $\Phi(D,f)$ is an invariant under Yoshikawa moves. If $D$ is a knotted surface diagram in $\mathbb{R}^3$, then 
$\Phi(D,f)$ is not changed by Roseman moves. Additionally, in all these cases, cohomologous cocycles give the same $\Phi(D,f)$.
\end{lemma}
\begin{proof}
The coboundary $\delta$ is defined via $(\delta f)(c)=f(\partial c)$. If $c_{D,\mathcal{C}}$ is a cycle assigned to a colored diagram before the Reidemeister, Yoshikawa, or Roseman move, and $c'_{D',\mathcal{C'}}$ is for the diagram after the move, then we have already proved that $c_{D,\mathcal{C}}$ and $c'_{D',\mathcal{C'}}$ are homologous. Thus, $c_{D,\mathcal{C}} - c'_{D',\mathcal{C'}}=\partial d$, for some $d$, and
\[
f(c_{D,\mathcal{C}})f(c'_{D',\mathcal{C'}})^{-1}=f(\partial d)=(\delta f)(d)=1.
\]
It follows that
\[
\Phi(D,f)=\sum_{\mathcal{C}}f(c_{D,\mathcal{C}})=\sum_{\mathcal{C'}}f(c_{D',\mathcal{C'}})=\Phi(D',f).
\]
If $f'f^{-1}=\delta g$, then
\[
f'(c_{D,\mathcal{C}})f(c_{D,\mathcal{C}})^{-1}=(\delta g)(c_{D,\mathcal{C}})=
g(\partial c_{D,\mathcal{C}})=g(0)=1.
\]
Therefore, $\Phi(D,f')=\Phi(D,f)$.
\end{proof}

\begin{example} \label{2el}
There are two two-element KTQs, $KTQ_{2,1}=(\{0,1\},T_1)$ and $KTQ_{2,2}=(\{0,1\},T_2)$, where 
\begin{align*}
& xyzT_1=x+y+z \pmod 2,\\
& xyzT_2=x+y+z+1 \pmod 2.
\end{align*}
Both are examples of knot-theoretic ternary groups (see \cite{NPZ18} for an algebraic analysis of such ternary groups, with an application to flat links; see also \cite{Nie14} for more general examples). Note that due to working modulo 2, for both KTQs the divisions $\eL$, $\M$, and $\R$ are the same as the primary operations.
We will now calculate $H_N^1(KTQ_{2,2};\mathbb{Z})$. Let $\chi$ denote the characteristic function
\begin{displaymath}
\chi_x(y)=\left\{\begin{array}{ll}
1&\textrm{if $x=y$}\\
0&\textrm{if $x\neq y$}.
\end{array} \right.
\end{displaymath} 
Then each n-cochain $f$ can be written as 
$f=\sum_{x\in X^{n+2}}C_x\chi_x$ for some coefficients $C_x$, where the sum ranges over non-degenerate $(n+2)$-tuples. Recall that the degenerate tuples are the ones containing $a$, $b$, $abb\R$ on three consecutive coordinates, for some $a$ and $b$, which, in case of $KTQ_{2,2}$, translates to 
\[
a,\ b,\ a+b+b+1=a+1.
\]
To find conditions for the coefficients $C_x$ in the case when $f$ is a 1-cocycle, we first look at the values of $f$ on boundaries of non-degenerate 4-tuples (degenerate triples are removed in calculations):
\begin{align*}
&f(\partial(0,0,0,0))=f(-(0,1,0)-(0,0,0)+(0,0,0)+(0,1,0))=0,\\
&f(\partial(1,1,1,1))=f(-(1,0,1)-(1,1,1)+(1,1,1)+(1,0,1))=0,\\
&f(\partial(0,1,0,1))=f(-(0,0,0)-(0,1,0)+(1,0,1)+(1,1,1))=0,\\
&f(\partial(1,0,1,0))=f(-(1,1,1)-(1,0,1)+(0,1,0)+(0,0,0))=0.\\
\end{align*}
It follows that 
\[
C_{(0,1,0)}+C_{(0,0,0)}-C_{(1,1,1)}-C_{(1,0,1)}=0,
\]
and if we set $x=C_{(0,0,0)}$, $y=C_{(0,1,0)}$, and $z=C_{(1,0,1)}$, then $f$ can be written as 
\begin{align*}
f&=x\chi_{(0,0,0)}+y\chi_{(0,1,0)}+z\chi_{(1,0,1)}+(x+y-z)\chi_{(1,1,1)}\\
&=x(\chi_{(0,0,0)}+\chi_{(1,1,1)})+y(\chi_{(0,1,0)}+\chi_{(1,1,1)})
+z(\chi_{(1,0,1)}-\chi_{(1,1,1)}).
\end{align*}
Coboundaries of the characteristic functions in lower dimension are as follows.
\begin{align*}
&\delta\chi_{(0,0)}=2\chi_{(0,0,0)}-2\chi_{(0,1,0)},\\
&\delta\chi_{(0,1)}=\chi_{(0,1,0)}+\chi_{(1,0,1)}-\chi_{(0,0,0)}-\chi_{(1,1,1)},\\
&\delta\chi_{(1,0)}=\chi_{(0,1,0)}+\chi_{(1,0,1)}-\chi_{(0,0,0)}-\chi_{(1,1,1)},\\
&\delta\chi_{(1,1)}=2\chi_{(1,1,1)}-2\chi_{(1,0,1)}.
\end{align*}
In $H_N^1(KTQ_{2,2};\mathbb{Z})$, we can write
\begin{align*}
f&=x(\chi_{(0,1,0)}+\chi_{(1,0,1)})+y(\chi_{(0,1,0)}+\chi_{(1,1,1)})
+z(\chi_{(1,0,1)}-\chi_{(1,1,1)})\\
&=(x+y)\chi_{(0,1,0)}+(x+z)\chi_{(1,0,1)}+(y-z)\chi_{(1,1,1)}\\
&=\alpha\chi_{(0,1,0)}+\beta\chi_{(1,0,1)}+(\alpha-\beta)\chi_{(1,1,1)}\\
&=\alpha(\chi_{(0,1,0)}+\chi_{(1,1,1)})-\beta(\chi_{(1,1,1)}-\chi_{(1,0,1)}).
\end{align*}
It follows that $H_N^1(KTQ_{2,2};\mathbb{Z})=\mathbb{Z}\times\mathbb{Z}_2$.
\end{example}

\subsection{KTQ cocycle invariants and cocycle invariants from the third quandle cohomology}
Knot-theoretic ternary quasigroups and quandles are universal-algebraic structures with a common root: the fundamental group of the knot complement. Quandle colorings in their basic version are colorings of arcs of a knot diagram, but they can be extended to colorings of regions, although in such case a color assigned to a given region determines the colors of all the other regions \cite{FR92,Kam02}.
Such extended colorings are called shadow colorings, and give cycles in the third quandle homology \cite{CKS01,KnSurf}. The following notion is important in our considerations.

\begin{figure}
\begin{center}
\includegraphics[height=2.8 cm]{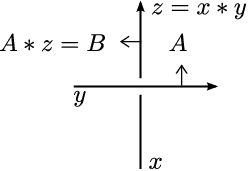}
\caption{The (shadow) quandle coloring rule.}\label{qcol}
\end{center}
\end{figure}

\begin{definition}[\cite{CKS01}]
Let $D$ be a knot diagram on a compact oriented surface $F$. Then the
fundamental shadow quandle $SQ(D)$ is defined using presentation as follows. The generators correspond to over-arcs and regions of $D$. The relations are defined for each
crossing as in ordinary fundamental quandles (Fig. \ref{qcol}), and at each arc dividing regions. Specifically, if $A$ and $B$ are generators corresponding to adjacent regions such that the normal points from
the region colored $A$ to that colored $B$, and if the arc dividing these regions is colored by $z$,
then $B = A*z$. 
\end{definition}
Two diagrams on $F$ that differ by Reidemeister moves on $F$ have isomorphic fundamental shadow quandles. The shadow colorings using a quandle $Q$ can be  regarded as quandle homomorphisms from the fundamental shadow quandle to $Q$ (see \cite{CKS01}). 

We note that it is possible to find examples of links on surfaces for which
$SQ(D)$ collapses (basically because of the first quandle axiom $a*a=a$, for all $a$), and $KTQ(D)$ and cohomological invariants retain useful information. Here is one such example.

\begin{figure}
\begin{center}
\includegraphics[height=4.5 cm]{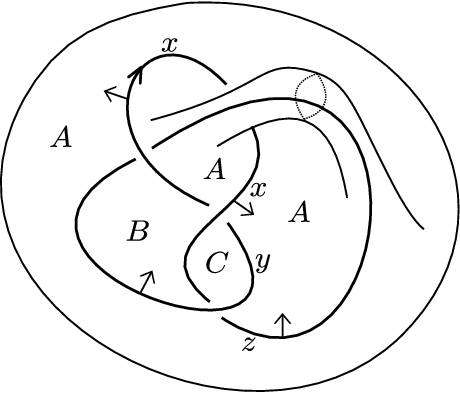}
\caption{A knot diagram on a torus with symbols for generators of $SQ(D)$ and $KTQ(D)$.}\label{8onT}
\end{center}
\end{figure}

\begin{example}
Consider the oriented knot diagram $D$ on a torus depicted in Fig. \ref{8onT}.
The symbols of the generators of $SQ(D)$ assigned to arcs are $x$, $y$, $z$, and the symbols of generators assigned to regions are $A$, $B$, and $C$. The quandle relation $x*x=y$ implies $x=y$, and 
$z*y=z*x=x$ is equivalent to $z=x\bar{*}x=x$, where $\bar{*}$ denotes the inverse quandle operation. Thus, $x=y=z$, and it follows that for any coloring with any quandle $Q$, the associated cycle in the third quandle homology (and the cocycle invariants) will be trivial. Indeed, the chain assigned to any colored crossing will have the colors of the under-arc and the over-arc on two consecutive coordinates, and since they are the same, it will be a degenerate chain. We can finish calculating 
$SQ(D)$: $A=A*x=B$, and $C=A*y=A*x=A$. Therefore, $SQ(D)$ can be presented as having two generators: $x$ and $A$, and one relation $A*x=A$. Such $SQ(D)$ is the same as $SQ$ for the meridian loop. $KTQ(D)$ has generators
$A$, $B$, $C$, and relations $A=AABT$, $A=ABCT$, and $B=AACT$. The signed colored paths are: $(A,A,B)$, $-(A,B,C)$, and $-(A,A,C)$. Let us use $KTQ_{2,2}$ with the cocycle $\phi=\chi_{(0,1,0)}+\chi_{(1,1,1)}$. There are two colorings of $D$ with 
$KTQ_{2,2}$. The first is given by the assignment $A\mapsto 0$, $B\mapsto 1$,
and $C\mapsto 0$, and the second by $A\mapsto 1$, $B\mapsto 0$,
and $C\mapsto 1$. Taking into consideration that for both colorings $(A,A,B)$ yields a degenerate chain, the cycles for the colored diagrams are
\[
c_1=-(0,1,0)-(0,0,0),\ \textrm{and}\ \, c_2=-(1,0,1)-(1,1,1).
\]
Therefore, $\Phi(D,\phi)=-t^2$, where $t$ denotes the generator of $\mathbb{Z}$ in the multiplicative notation.
This nontrivial value of the cocycle invariant distinguishes $D$ from some other knots on the torus, including the ones with crossingless diagrams. Note that
because the divisions in $KTQ_{2,2}$ are the same as the primary operation,
there is no difference between coloring of a crossing and the coloring of its mirror image. The distinction between crossings is made at the level of colored paths. An example that detects the presence of crossings in $D$ on the level of colorings is a three-element KTQ $(\{1,2,3\},T_3)$ that we describe with the primary operation multiplication cube. The multiplication cube can be sliced into the following three matrices, each matrix for a fixed first coordinate of $xyzT_3$.
\[ 
\begin{array}{|c| c ccc} 
1yzT_3&1 & 2 & 3 \\
\hline 
1 & 2 & 3 & 1 \\
2 & 3 & 1 & 2 \\
3 & 1 & 2 & 3 \\ 
\end{array}
\begin{array}{|c| c ccc} 
2yzT_3&1 & 2 & 3 \\
\hline 
1 & 1 & 2 & 3 \\
2 & 2 & 3 & 1 \\
3 & 3 & 1 & 2 \\ 
\end{array}
\begin{array}{|c| c ccc} 
3yzT_3&1 & 2 & 3 \\
\hline 
1 & 3 & 1 & 2 \\
2 & 1 & 2 & 3 \\
3 & 2 & 3 & 1 \\ 
\end{array}
\]
For example $123T_3=2$ and $231T_3=3$. 
There are no colorings of $D$ using this KTQ, and that distinguishes it from a diagram with no crossings, for which the number of colorings is $3^{|Reg|}$.
GAP calculations show that $H^N_0(X,T_3)=\mathbb{Z}$, $H^N_1(X,T_3)=\mathbb{Z}_3$,
and $H^N_2(X,T_3)=\mathbb{Z}_3^2$.
\end{example}

\subsection{Some computational examples}

Here we include more examples of calculations, with the help of GAP \cite{GAP4}. 

\begin{example}
Consider the following four-element KTQ.
\[ 
\begin{array}{|c| c cccc} 
1yzT&1 & 2 & 3 & 4 \\
\hline 
1 & 1 & 2 & 3 & 4 \\
2 & 2 & 1 & 4 & 3 \\
3 & 3 & 4 & 2 & 1 \\ 
4 & 4 & 3 & 1 & 2 \\
\end{array}
\begin{array}{|c| c cccc} 
2yzT&1 & 2 & 3 & 4 \\
\hline 
1 & 2 & 1 & 4 & 3 \\
2 & 1 & 2 & 3 & 4 \\
3 & 4 & 3 & 1 & 2 \\ 
4 & 3 & 4 & 2 & 1 \\
\end{array}
\begin{array}{|c| c cccc} 
3yzT&1 & 2 & 3 & 4 \\
\hline 
1 & 4 & 3 & 2 & 1 \\
2 & 3 & 4 & 1 & 2 \\
3 & 1 & 2 & 3 & 4 \\ 
4 & 2 & 1 & 4 & 3 \\
\end{array}
\begin{array}{|c| c cccc} 
4yzT&1 & 2 & 3 & 4 \\
\hline 
1 & 3 & 4 & 1 & 2 \\
2 & 4 & 3 & 2 & 1 \\
3 & 2 & 1 & 4 & 3 \\ 
4 & 1 & 2 & 3 & 4 \\
\end{array}
\]
This KTQ is not of the type $(X,T,\eL=T^{(2,3)},M=\hat{T},\R=T^{(1,2)})$, and its primary operation $T$ is not obtained by composing two binary quasigroup operations.
We calculated that $H^N_0(X,T)=\mathbb{Z}^2$, 
$H^N_1(X,T)=\mathbb{Z}^2\oplus \mathbb{Z}_4$,
and $H^N_2(X,T)=\mathbb{Z}^2\oplus \mathbb{Z}_2^2\oplus \mathbb{Z}_4^2$.

Let us consider a 1-cocycle with $\mathbb{Z}_2$ coefficients expressed as a sum of characteristic functions:
\begin{align*}
f&=\chi_{(1,1,4)}+\chi_{(1,2,2)}+\chi_{(1,2,3)}+\chi_{(1,4,2)}+\chi_{(2,1,1)}
+\chi_{(2,1,3)}+\chi_{(2,2,4)}+\chi_{(2,3,1)}\\
&+\chi_{(3,1,1)}+\chi_{(3,1,2)}+\chi_{(3,1,3)}+\chi_{(3,3,4)}
+\chi_{(4,2,1)}+\chi_{(4,2,2)}+\chi_{(4,2,4)}+\chi_{(4,3,1)}\\
&+\chi_{(4,3,2)}+\chi_{(4,4,1)}+\chi_{(4,4,2)}+\chi_{(4,4,3)}.
\end{align*}
Let $L_1$ denote the Hopf link (with the braid word $\sigma_1\sigma_1$),
$L_2$ be the Whitehead link ($\sigma_1\sigma_2^{-1}\sigma_1\sigma_2^{-1}\sigma_2^{-1}$), and let
$L_3$ denote the Borromean rings ($\sigma_1^{-1}\sigma_2\sigma_1^{-1}\sigma_2\sigma_1^{-1}\sigma_2$), with a generator $\sigma_i$ corresponding to a positive crossing.
Then the values of the cocycle invariant are $\Phi(L_1,f)=16+16t$, $\Phi(L_2,f)=32+32t$, and 
$\Phi(L_3,f)=64+192t$, where $t$ denotes the multiplicative generator of $\mathbb{Z}_2$.
\end{example}

\begin{example}
Using the KTQ $(\{1,2,3\},T_3)$ defined earlier, we note that
\[
g=\chi_{(1,3,1)}+\chi_{(1,3,2)}+\chi_{(2,1,2)}-\chi_{(2,2,3)}-\chi_{(2,3,3)}
+\chi_{(3,1,3)}+\chi_{(3,3,1)}
\]
is a 1-cocycle with $\mathbb{Z}_3$ coefficients. 
The value of the cocycle invariant for the square knot is $9+36t+36t^2$, and for the granny knot it is $45+18t+18t^2$. The value for the trefoil knot ($\sigma_1\sigma_1\sigma_1$) is $9+18t$, while for its mirror image it is
$9+18t^2$.
\end{example}

\begin{figure}
\begin{center}
\includegraphics[width=14 cm]{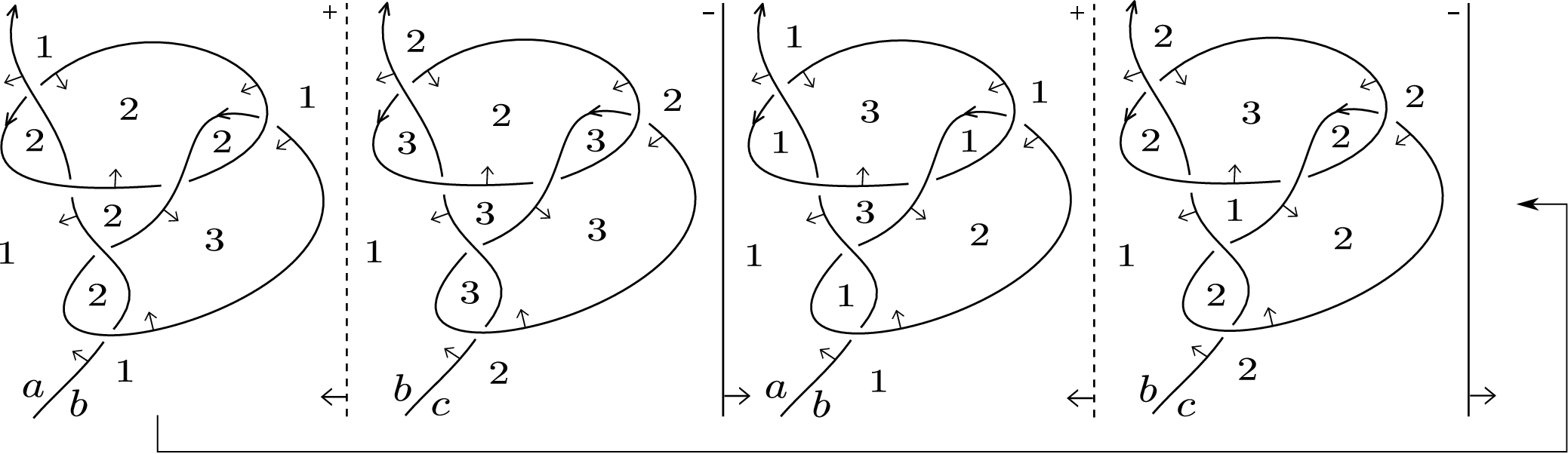}
\caption{A concise description of a KTQ coloring of a surface-link obtained by twist-spinning.}\label{L6a5dt}
\end{center}
\end{figure}

\begin{example} An important class of knotted surfaces are surfaces obtained by twist-spinning. An $r$-twist-spin of a knot $K$, denoted by $\tau^rK$, is a knotted sphere constructed by Zeeman in \cite{Zee65}. Similarly, one can obtain surface-links by twist-spinning classical links. A useful knotted surface diagram for $\tau^rK$ was described in \cite{AS05}, and our computations will be based on it. Because in KTQ colorings one colors only the regions in the complement of the diagram, it is rather easy to present the information required for the calculation of the cocycle invariants in a two-dimensional way. We will show it on an example of a $2$-twist-spin of a link L6a5 (in Thistlethwaite's notation). The resulting surface-link consists of a sphere and two tori.
In Fig. \ref{L6a5dt} each frame contains a KTQ-colored tangle $\mathcal{T}$, whose closure is the link L6a5. Referring to the description in \cite{AS05} (Section 3), the first and the third pictures correspond to $\mathcal{T}$ before it enters the first and the second spherical kink, respectively. The second and the fourth pictures of $\mathcal{T}$ show how its colorings look like inside the first and the second kink, respectively. After exiting the second kink, $\mathcal{T}$ has to have a coloring that is the same as the one in the first picture. There are $r$ kinks in the diagram of $\tau^rK$, so in general our coloring description would have $2r$ colored knot (or link) diagrams.
The separating lines represent the sheets belonging to the kinks. They are dashed if the kink is the bottom sheet for the triple points, and solid if it is the top sheet; the co-orientation is indicated by the arrows. For dashed lines, the signs of the triple points agree with the signs of the corresponding crossings in the preceding oriented tangle diagram, for solid lines they are opposite. For a coloring of such a surface-link to be valid, the colors $a$, $b$, and $c$ must repeat themselves in a pattern indicated in Fig. \ref{L6a5dt}, where $c$ is the color inside the kinks, $b$ is the color of the outside region, and $a$ is the color inside the sphere (near the axis of rotation). Additionally, two consecutive copies of each semi-arc of $\mathcal{T}$ (separated by a vertical line) correspond to a double point edge with its required relation. Fig. \ref{L6a5dt} shows a coloring with $(\{1,2,3\},T_3)$. Thus, for example, the relation for the pair of the top semi-arcs separated by the left-most dashed line is $211T_3=1$. 
The cycle assigned to the coloring from Fig. \ref{L6a5dt} before removing degenerate 4-tuples is:
\begin{align*}
&+(2,1,2,2)+(3,2,1,2)+(3,2,2,2)+(2,1,3,2)+ 
(3, 2, 3, 2)+(2, 1, 1, 2)\\
&-(2, 2, 3, 1)-(3, 1, 3, 1)-(3, 2, 3, 1)-(2, 3, 3, 1)-
(3, 3, 3, 1)-(2, 1, 3, 1)\\
&+(2, 1, 3, 1)+(1, 3, 1, 1)+(1, 3, 3, 1)+(2, 1, 2, 1)+(1, 3, 2, 1)+(2, 1, 1, 1)\\
&-(2, 3, 2, 2)-(1, 1, 2, 2)- 
(1, 3, 2, 2)-(2, 2, 2, 2)-(1, 2, 2, 2)-(2, 1, 2, 2). 
\end{align*}
After removing degenerate 4-tuples it becomes:
\begin{align*}
&+(3, 2, 1, 2)+(2, 1, 3, 2)+(3, 2, 3, 2)+(1, 3, 1, 1)+(2, 1, 2, 1)+(1, 3, 2, 1)\\
&-(3, 1, 3, 1)-(2, 3, 3, 1)-(3, 3, 3, 1)-(1, 1, 2, 2)-(2, 2, 2, 2)-(1, 2, 2, 2).
\end{align*}
The following is a $2$-cocycle with $\mathbb{Z}_3$ coefficients: 
\begin{align*}
\phi=&-\chi_{( 1, 1, 2, 2 )}-\chi_{(1, 2, 2, 3)}+\chi_{(1, 3, 1, 1 )}+
\chi_{(1, 3, 1, 3)}-\chi_{( 1, 3, 2, 3 )}-\chi_{( 2, 1, 2, 1)}+
\chi_{(2, 1, 3, 1)}\\
&-\chi_{(2, 1, 3, 2)}+\chi_{(2, 2, 2, 2 )}+\chi_{(2, 2, 2, 3 )}-\chi_{(2, 2, 3, 2 )}+\chi_{(2, 2, 3, 3)}+\chi_{(2, 3, 2, 3 )}+
\chi_{(2, 3, 3, 1 )}\\
&-\chi_{(2, 3, 3, 3)}+\chi_{(3, 1, 1, 1)}+\chi_{( 3, 1, 1, 2 )}-\chi_{(3, 1, 3, 1)}-\chi_{(3, 1, 3, 2)}-\chi_{(3, 2, 1, 3)}+
\chi_{(3, 2, 3, 3)}\\
&+\chi_{(3, 3, 1, 1)}+\chi_{(3, 3, 1, 3)}+\chi_{(3, 3, 3, 1)}-\chi_{(3, 3, 3, 3)}.
\end{align*}
Its evaluation on the above coloring cycle is 1 (or $t$, if we use the multiplicative notation). Taking into consideration all the colorings, the value of the cocycle invariant given by $\phi$ for this surface-link is $9+18t$.
If we change the orientation for all the components, the value of the invariant becomes $9+18t^2$. Thus, it is an example of a non-invertible surface-link.
\end{example}

\bibliography{simplex}
\bibliographystyle{plain}
\end{document}